\documentclass[hidelinks,letterpaper,reqno]{amsart}
\usepackage{graphicx}
\usepackage[margin = 1 in]{geometry}
\usepackage{subcaption}
\usepackage{tabularx}
\usepackage{booktabs}
\usepackage{array}
\usepackage{amsmath}
\usepackage{amsfonts}
\usepackage{amssymb}
\usepackage{amsthm}
\usepackage{thmtools}
\usepackage[scr]{rsfso}
\usepackage[normalem]{ulem}

\usepackage{permute}
\usepackage{slashed}

\usepackage[dvipsnames]{xcolor}

\usepackage{comment}
\usepackage{changepage}

\usepackage[T1]{fontenc}
\usepackage{makecell}

\usepackage{tikz}
\usetikzlibrary{calc}

\allowdisplaybreaks

\usepackage[utf8]{inputenc}
\usepackage[T1]{fontenc}

\usepackage[pagebackref = true, bookmarksdepth=2, linktocpage=true, hypertexnames=false]{hyperref}
\hypersetup{
	colorlinks   = true,
	citecolor    = Green,
	linkcolor    = Blue,
	urlcolor     = Blue
}

\usepackage{graphicx,amsmath,amssymb,amsthm,paralist,color,tikz-cd}
\usepackage{mathrsfs}
\usepackage{mathtools}
\usepackage{setspace}
\usepackage[color=blue!30]{todonotes}

\usepackage{amscd}
\usepackage{bbm}

\usepackage[capitalize]{cleveref}

\DeclareMathOperator{\interior}{int}

\DeclareMathOperator{\Id}{Id}
\DeclareMathOperator{\tr}{tr}

\DeclareMathOperator{\triv}{\mathbf{1}}

\DeclareMathOperator{\SO}{SO}

\DeclareMathOperator{\U}{U}
\DeclareMathOperator{\F}{F}

\DeclareMathOperator{\T}{T}

\DeclareMathOperator{\Ad}{Ad}

\DeclareMathOperator{\Sob}{\mathcal{S}}
\DeclareMathOperator{\Lie}{Lie}
\DeclareMathOperator{\Ind}{Ind}
\DeclareMathOperator{\compactsmooth}{\mathit{C}_{\mathrm{c}}^\infty}
\DeclareMathOperator*{\Hilbertoplus}{\widehat{\bigoplus}}

\newcommand{\N}{\mathbb{N}}
\newcommand{\Z}{\mathbb{Z}}

\newcommand{\R}{\mathbb{R}}
\newcommand{\C}{\mathbb{C}}

\newcommand{\LieG}{\mathfrak{g}}

\newcommand{\LieN}{\mathfrak{n}}

\newcommand{\LieK}{\mathfrak{k}}

\newcommand{\critexp}{\delta_\Gamma}

\newcommand{\BMS}{m^\mathrm{BMS}}
\newcommand{\Haar}{m^\mathrm{Haar}}
\newcommand{\BR}{m^\mathrm{BR}}
\newcommand{\BRstar}{m^{\mathrm{BR}_*}}

\newcommand{\bsl}{\backslash}

\newcounter{constC}

\Crefname{enumi}{}{}
\creflabelformat{enumi}{#2(#1)#3}
\Crefname{subsection}{Subsection}{Subsections}

\crefformat{enumi}{(#2#1#3)}

\newtheorem{thm}{Theorem}[section]
\newtheorem{lem}[thm]{Lemma}

\newtheorem{prop}[thm]{Proposition}

\newtheorem{cor}[thm]{Corollary}

\theoremstyle{definition}
\newtheorem{definition}[thm]{Definition}
\newtheorem*{definition-nono}{Definition}

\newtheorem*{conjecture-nonum}{Conjecture}

\theoremstyle{remark}
\newtheorem{remark}[thm]{Remark}

\newtheoremstyle{case}{}{}{}{}{}{:}{ }{}
\theoremstyle{case}

\numberwithin{equation}{section}

\renewcommand{\H}{\mathbb{H}}

\newcommand{\mc}{\mathcal}

\newcommand{\mrm}{\mathrm}

\newcommand{\G}{\Gamma}
\renewcommand{\d}{\delta}
\newcommand{\e}{\varepsilon}
\renewcommand{\l}{\lambda}

\newcommand{\s}{\sigma}

\renewcommand{\t}{\tau}

\renewcommand{\k}{\kappa}

\newcommand{\set}[1]{\left\{#1\right\}}
\renewcommand{\r}{\rightarrow}

\def\multiset#1#2{\ensuremath{\left(\kern-.3em\left(\genfrac{}{}{0pt}{}{#1}{#2}\right)\kern-.3em\right)}}

\newcommand{\norm}[1]{\left\lVert#1\right\rVert}

\newcommand{\supp}{\mathrm{supp}}

\newcommand{\Hcal}{\mc{H}}
\newcommand{\Ical}{\mc{I}}
\newcommand{\Jcal}{\mc{J}}

\newcommand{\Lcal}{\mc{L}}
\newcommand{\Mcal}{\mc{M}}

\newcommand{\Ucal}{\mc{U}}
\newcommand{\Vcal}{\mc{V}}

\newcommand{\Zcal}{\mc{Z}}

\newcommand{\Ghat}{\widehat{G}}
\newcommand{\GhatSD}{\widehat{G}_{\mrm{s}\textrm{-}\mrm{d}}}

\newcommand{\temp}{\mrm{temp}}
\newcommand{\qcomp}{\mrm{q}\textrm{-}\mrm{comp}}

\newcommand{\Mhat}{\widehat{M}}
\newcommand{\MhatSD}{\widehat{M}_{\mathrm{s}\textrm{-}\mathrm{d}}}

\newcommand{\GmodGamma}{\G\bsl G}
\newcommand{\HSmodGamma}{\G\bsl\H^{d + 1}}
\newcommand{\LtwoGmodGamma}{\Lcal}

\AtBeginDocument{%
	\def\MR#1{}
}

\begin{document}

	\title[Strong Spectral Gap]{Strong Spectral Gap for Geometrically Finite Hyperbolic Manifolds}
	\author{Dubi Kelmer}
	\address{Department of Mathematics, Boston College, Chestnut Hill, MA 02467, USA}
	\email{dubi.kelmer@bc.edu}
	\author{Osama Khalil}
	\address{Department of Mathematics, Statistics, and Computer Science, University of Illinois Chicago, Chicago, IL 60607, USA}
	\email{okhalil@uic.edu}
	\author{Pratyush Sarkar}
	\address{Department of Mathematics, ETH Zürich, 8092 Zürich, Switzerland}
	\email{psarkar@ethz.ch}

	\begin{abstract}
		Let $\G < G := \SO(d+1, 1)$ for $d \geq 1$ be a Zariski dense, geometrically finite, discrete subgroup with critical exponent strictly greater than $d/2$. We show that $L^2(\GmodGamma)$ admits a strong spectral gap, confirming a conjecture of Mohammadi and Oh. 
		This extends the spherical spectral gap on $L^2(\Gamma\backslash \H^{d+1}) \cong L^2( \GmodGamma / \SO(d+1))$, which follows by the works of Lax--Phillips, Patterson, and Sullivan by different methods.
		As a consequence, we establish rates of decay of matrix coefficients, and of exponential mixing of the frame flow, that are explicitly determined by the size of the strong spectral gap.
		
	\end{abstract}
	
	\maketitle
	
	\setcounter{tocdepth}{1}
	\tableofcontents
	
	\section{Introduction}
	\label{sec:Introduction}

	Throughout the paper, $G := \SO(d+1,1)^\circ$ denotes the connected component of the orthogonal group of signature $(d+1,1)$ for $d \geq 1$, and $\G < G$ is a discrete subgroup. Recall that $G$ can be realized as the group of orientation-preserving isometries of the $(d + 1)$-dimensional hyperbolic space $\H^{d + 1}$.
	The discrete subgroup $\Gamma$ and the corresponding hyperbolic orbifold $\G\bsl\H^{d + 1}$ are said to be \emph{geometrically finite} if the $1$-neighborhood of the convex core of $\G\bsl\H^{d + 1}$---the smallest closed convex subset containing all closed geodesics in $\G\bsl\H^{d + 1}$---is of finite volume.
	The \emph{critical exponent} $\critexp \in [0, d]$ is a fundamental parameter associated to $\Gamma$ which is defined as the abscissa of convergence of the Poincar\'e series $s \mapsto \sum_{\gamma \in \Gamma} e^{-sd(o, \gamma o)}$ for any reference point $o \in \H^{d + 1}$. When $\Gamma$ is geometrically finite, the critical exponent coincides with the Hausdorff dimension of the limit set of $\Gamma$ in $\partial_\infty(\H^{d + 1})$---the set of limit points of any $\Gamma$ orbit in $\overline{\H^{d + 1}}$. Our primary object of study is the right regular representation of $G$ on the space $L^2(\GmodGamma)$ of $\C$-valued $L^2$ functions with respect to a fixed $G$-invariant measure.
	
	Recall that an irreducible unitary representation of $G$ is called \textit{spherical} if it admits a non-zero fixed vector by a maximal compact subgroup of $G$.
	In 1982, Lax--Phillips showed in  \cite{LaxPhillips} that, when $\G\bsl\H^{d + 1}$ is geometrically finite, the intersection of the $L^2$ spectrum of the Laplace--Beltrami operator with $[0,d^2/4)$ consists of a (possibly empty) finite set of eigenvalues, each occurring with finite multiplicity (so-called \textit{small eigenvalues} of the Laplace--Beltrami operator). Equivalently, since a spherical vector in a spherical complementary series representation is an eigenfunction of the Laplace--Beltrami operator with eigenvalue in $[0,d^2/4)$,
	at most finitely many \emph{spherical} complementary series can occur in $L^2(\GmodGamma)$ and with finite multiplicities.
	Moreover, the works of Patterson \cite{Patterson} in the case $d=1$, and Sullivan \cite{Sullivan-DensityAtInfinity} in general show that there exists a bottom eigenvalue below the continuous spectrum if and only if $\critexp > d/2$ (cf. works of Elstrodt \cite{ElstrodtI,ElstrodtII,ElstrodtIII}); thereby yielding the existence of a spectral gap in this case, often called the Lax--Phillips spectral gap.

	In \cite[Definition 1.1]{MohammadiOh}, Mohammadi--Oh formulated a more general \emph{strong spectral gap} which includes the \textit{non-spherical} complementary series as well as the spherical ones (cf. \cite{KelmerSarnak,Kelmer} for similar types of questions for compact groups). It is natural to extend their definition to the following definition to account for all the quasi-complementary series, which exhaust the non-tempered elements of the unitary dual of $G$; see \cref{thm: non-tempered irreps}. 
	To formulate this definition, denote by $\MhatSD$ the set of isomorphism classes of self-dual irreducible representations of $M \cong \SO(d)$. Then, the quasi-complementary series are parametrized by pairs $(\sigma,s)$ for $\sigma\in \MhatSD$ and $s\in \Ical_\sigma$, where $\Ical_\sigma \subset (d/2, d]$ is an explicit interval determined by $\sigma$; see \cref{subsec: Unitary representation theory of G} for precise definitions. The corresponding series will be denoted $\Ucal(\sigma, s)$.

	\begin{definition}[Strong Spectral Gap]\label{def:SSG}
		We say that $L^2(\GmodGamma)$ admits a \emph{strong spectral gap} if both of the following hold.
		\begin{enumerate}
			\item\label{itm: non-spherical complementary series at critexp} For all non-trivial $\sigma \in \MhatSD$, the quasi-complementary series $\mathcal{U}(\sigma, \critexp)$ is not contained in $L^2(\GmodGamma)$.
			\item\label{itm: gap for complementary series} There exists $\eta > 0$ such that for all $\sigma \in \MhatSD$ and $s \in (\critexp - \eta, \critexp)$, the quasi-complementary series $\mathcal{U}(\sigma, s)$ is not weakly contained in $L^2(\GmodGamma)$.
		\end{enumerate}
	\end{definition}
	
	\begin{remark}
		\label{rem:SSG(1) Justification}
		Note that for all $\sigma \in \MhatSD$ and $s > \critexp$, $\mathcal{U}(\sigma, s)$ is not weakly contained in $L^2(\GmodGamma)$; see \cref{prop: no quasi-complementary series above critexp}. Therefore, provided \cref{def:SSG}\cref{itm: gap for complementary series} holds, for any non-trivial $\sigma \in \MhatSD$, a weak containment of $\mathcal{U}(\sigma, \critexp)$ in $L^2(\GmodGamma)$ is automatically a strong containment, which justifies the seemingly stronger requirement of \cref{def:SSG}\cref{itm: non-spherical complementary series at critexp}.
	\end{remark}
	
	In fact, from the Lax--Phillips spectral gap combined with the observation that the non-spherical quasi-complementary series do not exist for parameters $s > d - 1$, it follows that if $\G < G$ is a Zariski-dense geometrically finite discrete subgroup with critical exponent $\critexp > \max\{d/2, d - 1\}$, then $L^2(\GmodGamma)$ admits a strong spectral gap (cf. \cite[Theorem 3.27]{MohammadiOh}).
	In \cite[Conjecture 1.2]{MohammadiOh}, Mohammadi and Oh conjectured that the optimal condition $\critexp > d/2$ suffices for $L^2(\GmodGamma)$ to admit a strong spectral gap for the complementary series, in arbitrary dimension $d + 1 \geq 2$.
	The following is the main result of this article, which, in particular, confirms their conjecture.

	\begin{thm}
		\label{thm:SSG conjecture}
		If $\G < G$ is a Zariski-dense, geometrically finite, discrete subgroup with $\critexp > d/2$, then $L^2(\GmodGamma)$ admits a strong spectral gap.
	\end{thm}

	\begin{remark}
		We make the following observations:
		\begin{enumerate}
			\item According to the above discussion, \cref{thm:SSG conjecture} is new for $d \geq 3$.
			\item Restricting to the spherical complementary series, we obtain a different proof of the Lax--Phillips spectral gap; cf.\ \cref{subsec:RelationToPriorWorks}. 
			We also obtain a different proof of the aforementioned fact, due to Patterson and Sullivan, that the bottom eigenvalue of the $L^2$ spectrum of the Laplace--Beltrami operator is simple and given by $\critexp(d - \critexp)$; cf.\ \cref{prop:SSG conjecture part 1}.

			\item     The size of the strong spectral gap $\eta$ in \cref{def:SSG}\cref{itm: gap for complementary series} provided by \cref{thm:SSG conjecture} remains uniform over any family of subgroups $\G<G$ over which the rate of exponential mixing of the frame flows is uniform.
			In particular, such uniformity is known to hold for certain families of congruence subgroups of a fixed Zariski-dense convex cocompact subgroup $\G$ of an arithmetic group of $G$ by work of Oh--Winter for $d=1$ \cite{OhWinter-Congruence}, and by the third author for general $d\geq 1$  \cite{Sarkar-Congruence}.
			These results will be extended to the case $\G$ is geometrically finite in forthcoming work of the third author.

		\end{enumerate}
	\end{remark}

	Having established \cref{thm:SSG conjecture} on the existence of a strong spectral gap, we make the following definition.
	\begin{definition}[Strong Spectral Gap Parameter]
		\label{def:spec gap parameter}
		The \textit{strong spectral gap parameter} of $\Gamma$, denoted by $\k_\Gamma$, is the supremum over all $\eta \in (0, \critexp -d/2]$ for which \cref{def:SSG}\cref{itm: gap for complementary series} holds.
	\end{definition}

	\cref{thm:SSG conjecture} yields the following two results on rates of decay of matrix coefficients, and of exponential mixing of the frame flow. This extends the results of Edwards--Oh \cite{EdwardsOh}, who proved exponential mixing for the geodesic flow without relying on the existence of a strong spectral gap. For a brief survey of matrix coefficients, mixing, and related problems in homogeneous dynamics, we refer the reader to the article of Oh \cite[\S 7--8]{Oh-ICM} which will appear in the Proceedings of the ICM 2026.

	Denote by $\Haar$, $\BMS$, $\BR$, and $\BRstar$, compatibly chosen $G$-invariant, Bowen--Margulis--Sullivan, and Burger--Roblin measures; see \cref{sec:exp mixing} and references therein. We also denote by $\Sob^m_K(\GmodGamma)$ the $L^2$-Sobolev space of order $m > 0$ defined using derivatives along a fixed maximal compact subgroup $K < G$; see \cref{sec:Sobolev} for a precise definition.

	\begin{thm}
		\label{thm:upgraded decay of correlations for Haar}
		Let $\G < G$ be a Zariski-dense, geometrically finite, discrete subgroup with $\critexp > d/2$. Let $\k_0:=\min\{\k_\Gamma,1\}$, where $\k_\Gamma$ is the strong spectral gap parameter of $\Gamma$. There exists $m > d(d + 1)/2$ such that for all $\epsilon > 0$, and $\phi, \psi \in \Sob_K^m(\GmodGamma)$, and $t > 0$, we have
		\begin{multline*}
			e^{(d - \critexp)t}\int_{\GmodGamma} (\phi \circ a_t) \cdot \psi \, d\Haar =  \int_{\GmodGamma} \phi \, d\BR \int_{\GmodGamma} \psi \, d\BRstar
			+ O_\epsilon\bigl(e^{-(\k_0 - \epsilon)t}\|\phi\|_{\Sob^m_K(\GmodGamma)} \|\psi\|_{\Sob^m_K(\GmodGamma)}\bigr).
		\end{multline*}
	\end{thm}
	
	\begin{remark}
		\cref{thm:upgraded decay of correlations for Haar} strengthens the results of \cite{SarkarWinter,LiPanSarkar} when $\critexp > d/2$ in the following three aspects which are important for applications \cite{MohammadiOh,KelmerOh}:
		\begin{itemize}
			\item the exponential rate $\k_0$ of the error term is explicitly determined by the strong spectral gap parameter $\kappa_\Gamma$;
			\item the implicit constant does not depend on the supports of the test functions, in particular allowing non-compactly supported functions;
			\item regularity control on the test functions is weakened to only involve the Sobolev norm $\Sob^m_K$.
		\end{itemize}
		We note however that the results of \cite{SarkarWinter,LiPanSarkar} are used as input in our proof of \cref{thm:SSG conjecture}. In particular, \cref{thm:upgraded decay of correlations for Haar} does not give a different proof of these results.
	\end{remark}

	The following theorem follows from \cref{thm:upgraded decay of correlations for Haar} using the same proof as in \cite{MohammadiOh,OhWinter-Congruence,KelmerOh} which provide effective versions of Roblin's transverse intersection argument \cite{Roblin}.  Indeed, the explicit formula for the exponential rate is obtained from a generalization of \cite[Theorem 5.8]{OhWinter-Congruence} and its proof.
	
	\begin{thm}
		Let $\G < G$ be a Zariski-dense, geometrically finite, discrete subgroup with $\critexp > d/2$. Let $\k_1 := \frac{\k_0}{2(d + 3 + \k_0)}$, where $\k_0$ is the constant from \cref{thm:upgraded decay of correlations for Haar}.
		There exists $m > d(d + 1)/2$ such that for all $\phi, \psi \in C^m(\GmodGamma)$ and $t > 0$, we have
		\begin{align*}
			\int_{\GmodGamma} (\phi \circ a_t) \cdot \psi \, d\BMS = \int_{\GmodGamma} \phi \, d\BMS \int_{\GmodGamma} \psi \, d\BMS + O\bigl(e^{-\kappa_1 t}\|\phi\|_{C^m(\GmodGamma)} \|\psi\|_{C^m(\GmodGamma)}\bigr).
		\end{align*}
	\end{thm}

	\cref{thm:SSG conjecture} also gives the following immediate corollary.
	Recalling that the Casimir operator acts on the quasi-complementary series $\Ucal(\sigma,s)$ with the scalar $s(d-s)$ and that its action on $K$-invariant functions coincides with that of the Laplace--Beltrami operator, we have the following:
	
	\begin{cor}
		\label{cor:SSG_Casimir}
		The bottom eigenvalue $\critexp(d-\critexp)$ of the Laplace--Beltrami operator on $L^2(\HSmodGamma)$ is isolated from the rest of the spectrum of the Casimir operator on $L^2(\GmodGamma)$.
	\end{cor}
	
	\begin{remark}
		Note that the eigenvalue $\critexp(d-\critexp)$ of the Casimir operator may not be simple when $d$ is odd due to the presence of (tempered) discrete series representations for which the Casimir eigenvalue could be the same as that of a (non-tempered) quasi-complementary series representation, and these could also appear in $L^2(\GmodGamma)$. We note, however, that since the Casimir eigenvalues corresponding to discrete series form a discrete set there is no danger of these accumulating near $\critexp(d-\critexp)$.
		
	\end{remark}

	\subsection{Organization of the paper and outline of the proof}
	\label{sec:outline}

	Our strategy is based on the study of the Laplace transform of scaled matrix coefficients of test functions $\phi,\psi\in L^{2}(\GmodGamma)$ defined as follows:
	\begin{align}
		F(z) = \int_0^{+\infty} e^{-(z+\critexp- d)t} \langle \phi \circ a_t, \psi \rangle_{L^{2}(\GmodGamma)} \,dt \qquad \text{for all $z \in \C$ with $\Re(z) > 0$}.
	\end{align}

	After recalling the necessary representation-theoretic background in \cref{sec: Representation theoretic preliminaries}, we recall in \cref{sec: Rep theoretic asymptotic expansion of matrix coefficients} the work of Edwards and Oh \cite{EdwardsOh} on asymptotic expansions of matrix coefficients of complementary series representations, which we use to prove \cref{prop:Laplace transform rep theory side} providing precise information on  holomorphic continuations of Laplace transforms of such matrix coefficients.

	The main terms in the asymptotic expansions of matrix coefficients are expressed in terms of Harish-Chandra functions; \cref{def:HC C function}.
	In \cref{sec:nonvanishing}, we recall explicit formulas for the Harish-Chandra functions from \cite{EguchiKoizumiMamiuda}, which allow us to produce prescribed $K$-finite vectors for which the main terms in the Edwards--Oh expansions are non-vanishing; \cref{cor:nonvanishing}.
	We use such non-vanishing statements to provide obstructions of holomorphic extensions of the Laplace transform of matrix coefficients to regions of the real line parametrizing the quasi-complementary series that weakly occur in $L^2(\GmodGamma)$.
	
	In \cref{sec:exp mixing}, we recall the works of Winter and the third author \cite{SarkarWinter}, and Li, Pan, and the third author \cite{LiPanSarkar} on exponential mixing of the frame flow with respect to the Bowen--Margulis--Sullivan measure.
	Through Roblin's transverse intersection argument, this result provides a power-saving error term in the asymptotic formula for scaled matrix coefficients; \cref{thm:decay of correlations for Haar}.
	This result immediately implies a meromorphic continuation of the Laplace transform to a strip to the left of the imaginary axis, with at most one simple pole at the origin; \cref{cor:Laplace transform and exp mixing}.

	In \cref{sec:Laplace comparison}, we prove \cref{thm:SSG conjecture} by comparing the Laplace transform $F(z)$ computed in two different ways as above; namely using representation theory in \cref{prop:Laplace transform rep theory side} and dynamics in \cref{cor:Laplace transform and exp mixing}.
	To describe the idea, recall that $L^{2}(\GmodGamma)$ admits a direct integral decomposition over the unitary dual of $G$ against certain Borel measures, which we refer to as \textit{spectral measures}; cf.~\cref{sec: Representation theoretic preliminaries}.
	As a consequence, the main terms in Edwards--Oh's asymptotic expansions are given by suitable integrals against such measures.
	
	The key idea, due to Avila--Gou\"ezel \cite[\S 3]{AvilaGouezel}, is to interpret the representation-theoretic formula provided by \cref{prop:Laplace transform rep theory side} as a \emph{Stieltjes transform} (recalled in \cref{prop:Stieltjes inversion}) of a certain absolutely continuous measure with respect to the spectral measures. A notable feature of our setting is that the Radon--Nikodym derivatives of the measures in question are complex-valued in general, whereas in the case of $\mrm{SL}_2(\R)$-representations studied in \cite{AvilaGouezel}, such densities are real-valued and non-negative.
	One then combines the holomorphic continuation provided by dynamics in \cref{cor:Laplace transform and exp mixing}, with an extension of the Stieltjes inversion formula to complex-valued measures in \cref{cor: complex measure zero}, to get vanishing of the spectral measures on an open interval, with right endpoint the bottom of the spectrum of $L^2(\GmodGamma)$. This argument, carried out in \cref{prop:SSG conjecture part 2}, will verify \cref{def:SSG}\cref{itm: gap for complementary series}.
	
	Simplicity of the bottom of the spectrum stated in \cref{def:SSG}\cref{itm: non-spherical complementary series at critexp} follows by comparing residues of the Laplace transforms together with the observation that the main term on the dynamical side is a bilinear form of rank $1$ (being a product of integrals against Burger--Roblin measures). This argument is carried out in \cref{prop:SSG conjecture part 1}.

	\subsection{Relation to prior works}
	\label{subsec:RelationToPriorWorks}
	The overall strategy of combining exponential mixing with representation theory to obtain a spectral gap was employed before in the work of Avila, Gou\"ezel, and Yoccoz in the setting of $\mrm{SL}_2(\R)$ representations \cite[Appendix B]{AvilaGouzelYoccoz}, where the argument is attributed to Anantharaman, Bufetov, and Forni.
	Here, the idea is to consider \textit{spherical vectors}, i.e., those invariant by a maximal compact subgroup, in a given spherical complementary series where it is known that matrix coefficients are non-negative, and in fact are bounded \textit{below} up to uniform constants by $e^{-\l t}$, for a suitable $\l\geq 0$ depending on the complementary series in question.
	Together with upper bounds on matrix coefficients coming from exponential mixing, this quickly gives the desired spectral gap.

	The above argument in fact readily extends to higher dimensions to give a different proof of Lax--Phillips's \textit{spherical} spectral gap using exponential mixing of the \textit{geodesic flow}.
	To our knowledge, there are no analogs of such positivity and lower bounds for matrix coefficients for the non-spherical quasi-complementary series (i.e., those that do not admit non-zero spherical vectors).
	In our setting, the work of Edwards--Oh, together with the non-vanishing results in \cref{sec:nonvanishing}, substitute for the lack of such lower bounds.
	The lack of positivity introduces further subtleties, which are treated somewhat indirectly using complex analysis in the form of Stieltjes transforms as discussed above.
	
	Finally, we recall the work of Dyatlov and Guillarmou \cite{DyatlovGuillarmou-long,DyatlovGuillarmou-short} on meromorphic continuations of Laplace transforms of matrix coefficients of sections of vector bundles to the entire complex plane in the setting of Axiom A flows, which includes geodesic flows on convex cocompact manifolds.
	At least for convex cocompact subgroups,\footnote{
		A similar result was obtained for scalar-valued functions on convex cocompact quotients of negatively curved symmetric spaces by different methods in \cite{Khalil-MixingI}, which additionally yield meromorphic continuation for geometrically finite groups to an explicit strip to the left of the imaginary axis determined by $\critexp$ and the ranks of the cusps of $\Gamma$.} these results imply that, for any given $\sigma \in \MhatSD$, there are at most finitely many parameters $s\in (d/2,\critexp]$ for which the quasi-complementary series $\Ucal(\sigma,s)$ weakly occur in $L^{2}(\GmodGamma)$.
	However, due to the fact that $\MhatSD$ is infinite, such results do not imply the existence of a strong spectral gap in the sense of \cref{def:SSG}.

	\subsection*{Acknowledgments}
	We thank Semyon Dyatlov, Alex Kontorovich, Amir Mohammadi, and Amos Nevo for interesting discussions.
	We thank Hee Oh for suggesting that \cref{thm:upgraded decay of correlations for Haar} follows from our arguments, and for comments on an earlier version of the article.
	We also thank Sam Edwards for pointing out the existence of the ends of complementary series.
	O.K.\ acknowledges support under NSF grants  DMS-2337911 and DMS-2247713. P.S.\ acknowledges support under SNSF grant 10003145 during his time at ETH Z\"urich.
	
	\section{Representation theoretic preliminaries}
	\label{sec: Representation theoretic preliminaries}
	As stated in the introduction, $G := \SO(d+1,1)^\circ$ for $d \geq 1$ and $\G < G$ is a discrete subgroup with critical exponent $\critexp \in [0, d]$. We assume henceforth that $\G < G$ is a Zariski-dense geometrically finite discrete subgroup with $\critexp > d/2$. Correspondingly, $\Gamma\bsl\H^{d + 1}$ is a $(d+1)$-dimensional geometrically finite hyperbolic orbifold. In this section, we collect the necessary background on Lie algebras, Lie groups, and unitary representation theory.
	
	\subsection{Lie theory of \texorpdfstring{$G$}{G}}
	Let us recall the structure of $G$ and the main dynamical system associated to $\G$ for our purposes.
	
	We always use Fraktur letters to denote Lie algebras corresponding to Lie groups, e.g., $\LieG := \Lie(G)$. Fix a Cartan involution on $\LieG$. Via the Killing form on $\LieG$, it induces a left $G$-invariant and right $K$-invariant Riemannian metric on $G$. It also gives rise to an Iwasawa decomposition $G = KAN \cong K \times A \times N$ where
	\begin{align*}
		K &< G, & A &< G, & N &< G,
	\end{align*}
	is a maximal compact subgroup, a maximal diagonalizable subgroup, and a maximal horospherical subgroup, respectively. Let $M := Z_K(A) < K$, i.e., the centralizer of $A$ in $K$. We denote the Haar measure on $G$ induced by the Riemannian metric on $G$ by $\mu_G$ or $dg$ and for other induced unimodular topological groups similarly. Denote by $\T^1(\H^{d + 1})$ and $\F(\H^{d + 1})$ the unit tangent bundle and the oriented orthonormal frame bundle of $\H^{d + 1}$, respectively. We may make the following identifications as Riemannian manifolds equipped with left $G$-actions:
	\begin{align*}
		\H^{d + 1} &\cong G/K, & \T^1(\H^{d + 1}) &\cong G/M, & \F(\H^{d + 1}) &\cong G.
	\end{align*}
	Similarly, we may also make the following identifications as Riemannian orbifolds:
	\begin{align*}
		\Gamma\bsl\H^{d + 1} &\cong \GmodGamma/K, & \T^1(\Gamma\bsl\H^{d + 1}) &\cong \GmodGamma/M, & \F(\Gamma\bsl\H^{d + 1}) &\cong \GmodGamma.
	\end{align*}
	
	Recall that $G$ is of rank one meaning that $A$ is a one-parameter subgroup. We parametrize $A := \{a_t\}_{t \in \R}$ such that its right translation action on $G$ and $\GmodGamma$ (resp. $G/M$ and $\GmodGamma/M$) is the unit speed frame flow (resp. geodesic flow). We may further assume that the parametrization is such that $N$ is the \emph{contracting} horospherical subgroup; and let $\overline{N} < G$ be the \emph{expanding} horospherical subgroup. That is, we have
	\begin{align*}
		N &= \bigl\{g \in G: \lim_{t \to -\infty} a_tga_{-t} = e\bigr\}, & \overline{N} &= \bigl\{g \in G: \lim_{t \to +\infty} a_tga_{-t} = e\bigr\}.
	\end{align*}
	Moreover, $N$ and $\overline{N}$ (resp. $\LieN$ and $\overline{\LieN}$) are $d$-dimensional abelian Lie groups (resp. Lie algebras) isomorphic to $\R^d$. The above parametrization is equivalent to the following adjoint action of $a_t$:
	\begin{align*}
		\Ad(a_t)v &= e^t v, & \Ad(a_t)\overline{v} &= e^{-t}\overline{v}, \qquad \text{for all $v \in \LieN$, $\overline{v} \in \overline{\LieN}$, and $t \in \R$}.
	\end{align*}
	As a consequence of the above discussion, we have unique maps
	\begin{align*}
		\kappa&: G \to K, & H&: G \to \R, & n&: G \to N,
	\end{align*}
	such that
	\begin{align*}
		g = \kappa(g) a_{H(g)} n_g.
	\end{align*}
	Note that $H$ is then parameterized such that $H(a_t) = t$ for all $t \in \R$.

	\subsection{Unitary representation theory: generalities}
	\label{subsec: Unitary representation theory: generalities}

	Let us fix some notations and conventions which apply for any unimodular locally compact Hausdorff topological group, which we denote by $G$ only in this subsection. We refer the reader to the books \cite{Dixmier, WallachII,WarnerII} for details. Throughout the paper, all representations are \emph{strongly continuous} and over $\C$ without any specification. As such, a unitary representation of $G$ is a map
	\begin{align*}
		\pi: G \to \U(\Hcal)
	\end{align*}
	into the unitary group $\U(\Hcal)$ of a Hilbert space $\Hcal$ equipped with an inner product $\langle \cdot, \cdot \rangle_{\Hcal}$ and a corresponding norm $\|\cdot\|_{\Hcal}$, such that for all $v \in \Hcal$, the map $G \to \Hcal$ given by $g \mapsto \pi(g)v$ is continuous. For brevity, we often specify it as a unitary representation $(\pi, \Hcal)$ of $G$. We also use $\|\cdot\|_{\Hcal}$ for the operator norm of bounded operators on $\Hcal$.
	
	For any left (resp. right) $G$-action on a locally compact Hausdorff topological space $X$ endowed with a $G$-invariant Borel measure $\mu_X$, such as $G$ itself, we denote by $L^2(X)$ the Hilbert space equipped with the standard inner product $\langle \cdot, \cdot \rangle_{L^2(X)}$ defined by $\langle \phi, \psi\rangle_{L^2(X)} := \int_X \phi \cdot \overline{\psi} \, d\mu_X$ for all $\phi, \psi \in L^2(X)$. We also denote by $(\lambda_{G}, L^2(X))$ (resp. $(\rho_{G}, L^2(X))$) the left (resp. right) regular representation of $G$ which is defined by
	\begin{align*}
		\lambda_{G}(g)(\phi)(x) = \phi(g^{-1}x) \qquad \text{(resp. $\rho_{G}(g)(\phi)(x) = \phi(xg)$)}
	\end{align*}
	for all $x \in X$, $\phi \in L^2(X)$, and $g \in G$.
	
	A representation $(\rho, \Vcal)$ of $G$ is (strongly) contained in another representation $(\pi, \Hcal)$ of $G$ if $(\rho, \Vcal)$ is isomorphic to a subrepresentation of $(\pi, \Hcal)$. A unitary representation $(\rho, \Vcal)$ of $G$ is \emph{weakly} contained in another unitary representation $(\pi, \Hcal)$ of $G$ if
	any diagonal matrix coefficient of $(\rho, \Vcal)$ is a uniform limit on compact subsets of $G$ of a sequence of linear combinations of diagonal matrix coefficients of $(\pi, \Hcal)$. Recall that a representation of $G$ is said to be irreducible if it does not strictly contain any non-trivial $G$-invariant subspace. The \emph{unitary dual} of $G$, denoted by $\Ghat$, is the set of all isomorphism classes of \emph{irreducible} unitary representations equipped with the hull-kernel/Fell topology. We denote by $\triv \in \Ghat$ the equivalence class of trivial irreducible representations.
	
	The \emph{dual} of a representation $(\pi, \Hcal)$ of $G$ is a new representation $(\pi^*, \Hcal^*)$ of $G$ defined by
	\begin{align*}
		\pi^*(g)(f)(v) = f(\pi(g^{-1})v)
	\end{align*}
	for all $v \in \Hcal$, $f \in \Hcal^*$, and $g \in G$. A representation, $(\pi, \Hcal)$, is \emph{self-dual} if it is isomorphic to $(\pi^*, \Hcal^*)$. We denote $\GhatSD \subset \Ghat$ for the subset of self-dual elements.
	
	A unitary representation $(\pi, \Hcal)$ of $G$ is said to be \emph{tempered} if its matrix coefficients are in $L^{2 + \epsilon}(G)$ for all $\epsilon > 0$. We have the decomposition into a disjoint union $\Ghat = \Ghat_{\temp} \sqcup \Ghat_{\qcomp}$, where $\Ghat_{\temp}$ is the subset consisting of tempered elements and $\Ghat_{\qcomp}$ is the subset consisting of the non-tempered elements.
	
	Let $(\pi, \Hcal)$ be a unitary representation of $G$. We cover some decompositions (see \cite[Chapter 14, \S 14.10, Theorem 14.10.5]{WallachII}). There exists an associated Borel spectral measure $m^{\Hcal}$ on $\Ghat$, which can be further decomposed into Borel spectral measures $m^{\Hcal}_\bullet$ on $\Ghat_\bullet$ for $\bullet \in \{\temp, \qcomp\}$ such that we have the following orthogonal direct integral decomposition:
	\begin{align*}
		(\pi, \Hcal) &= \int_{\Ghat}^\oplus  (\pi_\xi, \Hcal_\xi) \, dm^{\Hcal}(\xi) \\
		&= (\pi_{\temp}, \Hcal_{\temp}) \oplus (\pi_{\qcomp}, \Hcal_{\qcomp}) \\
		&= \int_{\Ghat_{\temp}}^\oplus (\pi_\xi, \Hcal_\xi) \,dm^{\Hcal}_{\temp}(\xi) \oplus \int_{\Ghat_{\qcomp}}^\oplus (\pi_\xi, \Hcal_\xi) \,dm^{\Hcal}_{\qcomp}(\xi)
	\end{align*}
	where $(\pi_{\temp}, \Hcal_{\temp})$ is the tempered subrepresentation and $(\pi_{\qcomp}, \Hcal_{\qcomp})$ is its orthogonal complement, and $(\pi_\xi, \Hcal_\xi)$ is isotypic of type $\xi = (\pi_\xi, \widehat{\Hcal}_\xi) \in \Ghat$, i.e., abusing notation we have the isomorphism $(\pi_\xi, \Hcal_\xi) \cong (\pi_\xi, \widehat{\Hcal}_\xi) \widehat{\otimes} (\triv, \Mcal_\xi)$ for some (possibly infinite dimensional) multiplicity Hilbert space $\Mcal_\xi$. Recall that atoms of Borel measures must be singletons. Moreover, we have the following well-known properties (see \cite[Chapter 14, \S 14.10, Lemmas 14.10.6 and 14.10.7]{WallachII}).
	
	\begin{lem}
		\label{lem: m strong/weak containment}
		Let $(\pi, \Hcal)$ be a unitary representation of $G$. The following holds.
		\begin{enumerate}
			\item An element $\xi \in \Ghat$ is strongly contained in $(\pi, \Hcal)$ if and only if $\xi$ is an atom of $m^\Hcal$.
			\item An element $\xi \in \Ghat$ is weakly contained in $(\pi, \Hcal)$ if and only if $\xi \in \supp(m^\Hcal)$.
		\end{enumerate}
	\end{lem}
	
	\begin{lem}
		\label{lem: m has countable atoms}
		Let $(\pi, \Hcal)$ be a unitary representation of $G$. Suppose that $\Hcal$ is separable. Then, $m^\Hcal$ has at most a countable number of atoms.
	\end{lem}
	
	Now, let $L < G$ be a compact subgroup. Since we may view $(\pi, \Hcal)$ as a unitary representation of $L$, i.e., $(\pi|_L, \Hcal)$, we also have the decomposition
	\begin{align*}
		(\pi|_L, \Hcal) = \Hilbertoplus_{\t\in \widehat{L}} (\t, \Hcal_\t).
	\end{align*}
	where $(\t, \Hcal_\t)$ is isotypic of type $\tau \in \widehat{L}$ in the sense as above---we often call them $L$-types. We also say that $\tau \in \widehat{L}$ is contained in $(\pi, \Hcal)$ whenever $\Hcal_\tau$ is non-trivial and write $\tau \subset \pi$. The classification of $\tau \in \widehat{L}$ which are contained in $(\pi, \Hcal)$ are called \emph{branching laws/rules}. Let $\tau \in \widehat{L}$. Recall that there is an associated character $\chi_\tau \in L^2(L)$ defined by $\chi_\tau(l) = \tr(\tau(l))$ and an operator $\mathsf{P}_\t^{(\pi, \Hcal)}: \Hcal \to \Hcal$ defined by
	\begin{align}\label{eq:definition of P_tau}
		\mathsf{P}_\t^{(\pi, \Hcal)} v = \dim(\tau) \int_L \overline{\chi_\tau(l)} \cdot \pi(l)v \, dl \qquad \text{for all $v \in \Hcal$},
	\end{align}
	where $dl$ is the Haar probability measure on $L$.
	Then it is in fact an orthogonal projection operator onto $\mathsf{P}_\t^{(\pi, \Hcal)}(\Hcal) = \Hcal_\t$. For a maximal compact subgroup $K < G$, the unitary representation $(\pi, \Hcal)$ is \emph{spherical} if it contains a non-trivial $K$-invariant vector.

	\subsection{Sobolev norms}\label{sec:Sobolev}
	Given a  unitary representation $(\pi,\Hcal)$ of $K$,
	a fixed basis $\{X_j\}_{j = 1}^{\dim(\LieK)}$ of $\LieK$, and $m\in \Z_{\geq 0}$, the Sobolev norm $\|\cdot\|_{\Sob_K^m(\Hcal)}$ is defined on smooth vectors $v \in \Hcal^\infty \subset \Hcal$ as
	$$\|v\|^2_{\Sob_K^m(\Hcal)}=\|v\|_\Hcal^2+\sum_{U} \|d\pi(U)(v)\|^2_\Hcal,$$
	where the sum is over all monomials $U$ of degree $m$ in $X_1,\ldots, X_{\dim(\LieK)}$.
	Let $\Sob_K^m(\Hcal)$ denote the completion of $\bigl\{v\in \Hcal^\infty: \|v\|_{\Sob_K^m(\Hcal)}< +\infty\bigr\} \subset \Hcal$ with respect to the Sobolev norm.
	While the Sobolev norm depends on the choice of basis of $\LieK$, different choices give rise to equivalent norms. We note that when $\Hcal$ is finite dimensional, $\Sob_K^m(\Hcal)=\Hcal$.

	Given a unitary representation $(\pi,\Hcal)$ of $G$ we define the Sobolev norm on $\Hcal$ by viewing $(\pi|_K, \Hcal)$ as a unitary representation of $K$.
	It is well known that $K$-finite vectors are smooth and hence have finite Sobolev norm.
	Moreover, we note that for vectors of a fixed $K$-type the Sobolev norm is equivalent to the $L^2$ norm. Explicitly, we have the following.

	\begin{lem}\label{lem: Sobolev bound}
		Let $(\pi,\Hcal)$ be a unitary representation of $G$ on a separable Hilbert space $\Hcal$. For any $\tau\in \widehat{K}$ there is a constant $C(\tau) > 0$ such that for any $\phi\in \Hcal_\tau$, we have
		$$\|\phi\|_{\Sob_K^1(\Hcal)}\leq C(\tau) \|\phi\|_{\Hcal}.$$
	\end{lem}
	
	\begin{proof}
		Let $V_\tau$ denote the finite dimensional Hilbert space realizing $\tau\in \widehat{K}$ and note that as $K$-representations, we have $(\tau, \Hcal_\tau) \cong (\tau, V_\tau) \otimes (\triv, \Mcal_\tau)$, where $K$ acts trivially on a (possibly infinite dimensional) separable multiplicity space $\Mcal_\tau$. That is,
		$$\pi(k) (u\otimes v)=(\tau(k)u)\otimes v \qquad \text{for all $u\otimes v\in V_\tau\otimes \Mcal_\tau$ and $k\in K$}.$$
		Let $\{v_j\}_{j=1}^n$ for some $n \in \Z_{\geq 0} \sqcup \{\infty\}$ be a countable orthonormal Hilbert basis for $\Mcal_\tau$. Let $\phi\in \Hcal_\tau\cong V_\tau\otimes \Mcal_\tau$. Note that it has a decomposition $\phi=\sum_{j=1}^n \phi_j\otimes v_j$, with $\phi_j\in V_\tau$ and $\|\phi\|_\Hcal^2=\sum_{j=1}^n \|\phi_j\|_{V_\tau}^2 < +\infty$.
		Letting $X\in \LieK$, we have
		$$d\pi(X)\phi=\sum_{j=1}^n (d\tau(X)\phi_j)\otimes v_j,$$
		so that by orthonormality of the basis $\{v_j\}_{j=1}^n$, we have
		\begin{align*}
			\|d\pi(X)\phi\|_\Hcal^2&= \left\langle \sum_{i=1}^n(d\tau(X)\phi_i)\otimes v_i, \sum_{j=1}^n (d\tau(X)\phi_j)\otimes v_j \right\rangle_\Hcal \\
			&=\sum_{i,j=1}^n\langle d\tau(X)\phi_i, d\tau(X)\phi_{j}\rangle_{V_\tau} \cdot \langle v_i, v_j\rangle_{\Mcal_\tau} \\
			&=\sum_{j=1}^n\| d\tau(X)\phi_j\|_{V_\t}^2.
		\end{align*}
		Since $V_\tau$ is finite dimensional, for any $v\in V_\tau$ we can bound $\|d\tau(X)v\|_{V_\tau}\leq \|d\tau(X)\|_{V_\tau} \|v\|_{V_\tau}$ where $\|d\tau(X)\|_{V_\tau}$ is the operator norm.
		Hence,
		\begin{align*}
			\|d\pi(X)\phi\|_\Hcal^2&\leq \|d\tau(X)\|_{V_\tau} \sum_{j=1}^n\|\phi_j\|_{V_\tau}^2=\|d\tau(X)\|_{V_\tau} \|\phi\|_\Hcal^2.
		\end{align*}
		Finally, for the fixed basis $\{X_j\}_{j = 1}^{\dim(\LieK)}$ of $\LieK$, taking
		$C(\tau)=1+\max_{1 \leq j \leq \dim(\LieK)} \|d\tau(X_j)\|_{V_\tau}$ gives the desired bound for $\|\phi\|_{\Sob_K^1(\Hcal)}$ and finishes the proof.
	\end{proof}

	\subsection{Unitary representation theory of \texorpdfstring{$\SO(n)$}{SO(n)}}
	\label{subsec: Unitary representation theory of SO}
	We recall some fundamental facts regarding irreducible representations of $\SO(n)$ for any $n \geq 1$, which then apply to both $K \cong \SO(d + 1)$ and its subgroup $M \cong \SO(d)$ (see for instance \cite{Thieleker}).
	
	Firstly, irreducible representations of $\SO(n)$ are unitarizable by compactness of $\SO(n)$, and consequently form the unitary dual $\widehat{\SO(n)}$. By the Peter--Weyl theorem, they are finite-dimensional and contained in $\bigl(\lambda_{\SO(n)}, L^2(\SO(n))\bigr)$; in fact
	\begin{align*}
		\bigl(\lambda_{\SO(n)}, L^2(\SO(n))\bigr) = \Hilbertoplus_{\tau \in \widehat{\SO(n)}} (\tau, V_\tau)^{\oplus \dim(\tau)}.
	\end{align*}
	\subsubsection{Case $2 \mid n$} Let us write $n = 2m$ for $m \in \Z_{> 0}$. The unitary dual $\widehat{\SO(2m)}$ can be canonically identified with
	\begin{align*}
		\{(\tau_1, \dotsc, \tau_m) \in \Z^m: \tau_1 \geq \tau_2 \geq \dotsb \geq \tau_{m - 1} \geq |\tau_m|\}
	\end{align*}
	as topological spaces. As such, we write $\tau = (\tau_1, \dotsc, \tau_m) \in \widehat{\SO(2m)}$ to specify an irreducible representation where the constraints on the integers are understood. Note that, if $4 \mid n$ or equivalently $2 \mid m$, then $\tau$ is self-dual; otherwise, $\tau^* = (\tau_1, \dotsc, \tau_{m - 1}, -\tau_m)$. For the moment, let us view $\SO(2m - 1) \subset \SO(2m)$ as a subgroup and let $\sigma = (\sigma_1, \dotsc, \sigma_{m - 1}) \in \widehat{\SO(2m - 1)}$ (see Case $2 \nmid n$ below). Then, the associated branching law says that $\sigma$ is contained in $\tau$ if and only if they satisfy the interlacing property
	\begin{align*}
		\t_1 \geq \sigma_1 \geq \t_2 \geq \sigma_2 \geq \dotsb \geq \t_{m - 1} \geq \sigma_{m - 1} \geq |\tau_m|;
	\end{align*}
	and in which case the multiplicity is $1$.
	
	\subsubsection{Case $2 \nmid n$} Let us write $n = 2m + 1$ for $m \in \Z_{\geq 0}$. The unitary dual $\widehat{\SO(2m + 1)}$ can be canonically identified with
	\begin{align*}
		\{(\tau_1, \dotsc, \tau_m) \in \Z^m: \tau_1 \geq \tau_2 \geq \dotsb \geq \tau_m \geq 0\}
	\end{align*}
	as topological spaces; for $m = 0$, we use the convention that the latter is a singleton containing the empty tuple. Again, we write $\tau = (\tau_1, \dotsc, \tau_m) \in \widehat{\SO(2m + 1)}$ to specify an irreducible representation. Note that $\tau$ is always self-dual. Again for the moment, let us view $\SO(2m) \subset \SO(2m + 1)$ as a subgroup and let $\sigma = (\sigma_1, \dotsc, \sigma_m) \in \widehat{\SO(2m)}$ (see Case $2 \mid n$ above). Then, the associated branching law says that $\sigma$ is contained in $\tau$ if and only if they satisfy the interlacing property
	\begin{align*}
		\t_1 \geq \sigma_1 \geq \t_2 \geq \sigma_2 \geq \dotsb \geq \t_{m - 1} \geq \sigma_{m - 1} \geq \tau_m \geq |\sigma_m|;
	\end{align*}
	and in which case the multiplicity is $1$.
	
	\subsection{Unitary representation theory of \texorpdfstring{$G$}{G}}
	\label{subsec: Unitary representation theory of G}
	The necessary background covered in this subsection is contained in \cite{Knapp} with further details in references therein. Working in the \emph{compact picture} (in the terminology of \cite[Chapter VII, \S 1]{Knapp}), for all $s \in \C$, define the representation $(U^s, L^2(K))$ of $G$ by
	\begin{align*}
		U^s(g)(\phi)(k) = e^{-sH(g^{-1}k)} \phi(\kappa(g^{-1}k)) \qquad \text{for all $k \in K$, $\phi \in L^2(K)$, and $g \in G$}.
	\end{align*}
	Though the latter is not obvious in the compact picture, we have the following facts:
	\begin{itemize}
		\item $(U^s|_K, L^2(K)) = (\lambda_K, L^2(K))$ is the left regular representation of $K$;
		\item $(U^s, L^2(K))$ is a \emph{unitary} representation if and only if $s \in d/2 + i\R$.
	\end{itemize}
	
	Recall that $P := MAN \cong M \times A \times N$. For any $s \in \C$, let $\chi_s: A \to \C$ be the character defined by $\chi_s(a_t) = e^{st}$ for all $t \in \R$. For any $\sigma \in \Mhat$ and $s \in \C$, viewing $\sigma \otimes \chi_s \otimes \triv$ as a representation of $P$, we may consider the induced representation $\Ind_P^G(\sigma \otimes \chi_s \otimes \triv)$;
	this is called parabolic induction (see \cite[Chapter VII, \S 1]{Knapp}). If $s \in i\R$, they are unitary representations called the \emph{principal series}, and they are irreducible if $s \neq 0$ (see \cite[Theorem 5]{KnappStein}). More generally, the isomorphism classes of the so-called irreducible admissible representations of $G$ (strictly containing the unitary dual $\Ghat$) arise as unique quotients of $\Ind_P^G(\sigma \otimes \chi_s \otimes \triv)$ with $\sigma \in \Mhat$ and $s \in \C$; these are called the Langlands quotients. Moreover, they are tempered only if $s \in i\R$. We refer the reader to \cite[Chapter XIV, \S 17, Theorem 14.92]{Knapp}, \cite[Chapter VIII, \S 14, Theorem 8.53]{Knapp}, and \cite{Langlands} for more details. For all $\sigma \in \Mhat$, define
	\begin{align*}
		L^2(K: \sigma) := \mathsf{P}_\sigma^{(\rho_K|_M, L^2(K))}(L^2(K)) = \left\{\phi \in L^2(K): \dim(\sigma)\int_M \overline{\chi_\sigma(m)} \cdot \rho_K(m)\phi \, dm = \phi\right\}.
	\end{align*}
	It turns out that for all $\sigma \in \Mhat$ and $s \in \C$, we may identify
	\begin{align}
		\label{eqn: LtwoKsigma as induced representation}
		(U^s, L^2(K: \sigma)) \cong \Ind_P^G(\sigma \otimes \chi_{s - d/2} \otimes \triv)^{\oplus \dim(\sigma)}.
	\end{align}
	
	We now introduce an intertwining operator which plays a central role in the classification of $\Ghat$ \cite{KunzeStein,KnappStein,Knapp}. Recall that $\MhatSD \subset \Mhat$ denotes the subset of self-dual elements. Let $\sigma \in \MhatSD$. For all $s > d/2$, define a bounded operator $\mathcal{A}(\sigma, s): L^2(K : \sigma) \to L^2(K : \sigma)$ by
	\begin{align*}
		\mathcal{A}(\sigma, s) = \sigma(w_0) \gamma_\sigma(s)^{-1} \int_{\overline{N}} e^{-s H(\overline{n})} \rho_G(w_0\k(\overline{n})^{-1}) \, d\overline{n}
	\end{align*}
	where $w_0 \in N_K(A)/M$ is the unique nontrivial Weyl element, $\sigma(w_0): L^2(K : \sigma) \to L^2(K : \sigma)$ is a unitary operator according to the unique extension (up to sign) of $\sigma$ from $M$ to $N_K(A)$, and $\gamma_\sigma: \C \to \C$ is a certain meromorphic normalizing function. It enjoys the intertwining property
	\begin{align}
		\label{eqn: A intertwining property}
		\mathcal{A}(\sigma, s)U^s(g) = U^{-s}(g)\mathcal{A}(\sigma, s) \qquad \text{for all $g \in G$}.
	\end{align}
	Abusing notation, we also denote by $\mathcal{A}(\sigma, s)$ the induced operator on $\Ind_P^G(\sigma \otimes \chi_{s - d/2} \otimes \triv)$ via any isomorphic subrepresentation of $L^2(K : \sigma)$.
	
	Let $\sigma \in \Mhat$. Writing $\sigma = (\sigma_1, \dotsc, \sigma_{\lfloor d/2\rfloor})$, define $1 \leq \ell(\sigma) \leq \lfloor d/2\rfloor$ to be the largest index with a positive integer coordinate and put $\ell(\sigma)=0$ if $\sigma$ is the trivial irreducible representation $\triv\in\Mhat$.
	Define the corresponding half-closed interval
	\begin{align*}
		\Ical_\sigma :=
		(d/2, d - \ell(\sigma)] \subset (d/2, d].
	\end{align*}
	The following is a theorem of Knapp--Stein \cite[Theorem 6, and Propositions 49 and 50]{KnappStein}.
	
	\begin{thm}
		Let $\sigma \in \Mhat$ and $s \in \R$. Then, $(U^s, L^2(K : \sigma))$ is \emph{unitarizable} with a semidefinite inner product if and only if $\sigma$ is self-dual and $s \in \Ical_\sigma$.
		Additionally, such an inner product is positive definite if and only if $s \in \interior\Ical_\sigma$.
		The semidefinite inner product $\langle \cdot, \cdot \rangle_{(\sigma, s)}$ on $L^2(K : \sigma)$ can be defined by
		\begin{align}\label{eq:intertwining}
			\langle u, v \rangle_{(\sigma, s)} = \langle \mathcal{A}(\sigma, s)u, v\rangle_{L^2(K)} \qquad \text{for all $u, v \in L^2(K : \sigma)$}.
		\end{align}
	\end{thm}
	
	\begin{remark}
		Indeed, unitarity with respect to the (possibly semidefinite) inner product follows from the intertwining property of $\mathcal{A}(\sigma, s)$.
	\end{remark}
	
	Let $\sigma \in \MhatSD$ and $s \in \R$. Denote by
	\begin{align*}
		\Zcal(\sigma, s) := \{v \in L^2(K : \sigma): \|v\|_{(\sigma, s)} = 0\} \subset L^2(K : \sigma)
	\end{align*}
	the subspace of null vectors in $L^2(K : \sigma)$ for the seminorm $\|\cdot\|_{(\sigma, s)}$ corresponding to the semidefinite inner product $\langle \cdot, \cdot \rangle_{(\sigma, s)}$, which is nontrivial if and only if $s = d - \ell(\sigma) \in \partial\Ical_\sigma$. Then, observe that $(U^s, \ker\mathcal{A}(\sigma, s)) = (U^s, \Zcal(\sigma, s))$ and is a subrepresentation of $(U^s, L^2(K : \sigma))$.
	
	We have the following theorem according to the classification results of Hirai \cite{Hirai}. We also refer the reader to \cite[Pg 152]{KnappZuckerman} for a deduction of this result from the general results of Langlands \cite{Langlands} and Knapp--Stein \cite{KnappStein}.
	
	\begin{thm}
		\label{thm: non-tempered irreps}
		The non-tempered elements of $\Ghat$ are exhausted by the \emph{quasi-complementary series} which is itself the union of the two families in the following list:
		\begin{enumerate}
			\item the \emph{complementary series} consisting of
			\begin{align*}
				(\pi_{\sigma, s}, \Ucal(\sigma, s)) := \Ind_P^G(\sigma \otimes \chi_{s - d/2} \otimes \triv) \qquad \text{for $\sigma \in \MhatSD$ and $s \in \interior\Ical_\sigma$};
			\end{align*}
			\item the \emph{ends of complementary series} consisting of
			\begin{align*}
				(\pi_{\sigma, s}, \Ucal(\sigma, s)) := \Ind_P^G(\sigma \otimes \chi_{s - d/2} \otimes \triv)/\ker\mathcal{A}(\sigma, s) \qquad \text{for $\sigma \in \MhatSD$ and $s = d - \ell(\sigma) \in \partial\Ical_\sigma$};
			\end{align*}
		\end{enumerate}
		and in both cases equipped with an inner product $\langle \cdot, \cdot \rangle_{\mathcal{U}(\sigma, s)}$ induced by the (possibly semidefinite) inner product $\langle \cdot, \cdot \rangle_{(\sigma, s)}$ on $L^2(K : \sigma)$.
	\end{thm}
	
	\begin{remark}
		Indeed, since $\ker\mathcal{A}(\sigma, s) = \Zcal(\sigma, s)$, the inner product $\langle \cdot, \cdot \rangle_{\Ucal(\sigma, s)}$ is positive-definite for the ends of complementary series $\Ucal(\sigma, s)$ for $\sigma \in \MhatSD$ and $s = d - \ell(\sigma) \in \partial\Ical_\sigma$.
	\end{remark}
	
	When $\sigma = \triv \in \MhatSD$, then there exists a $K$-invariant vector and so these are called the spherical quasi-complementary series. The remaining ones are called the non-spherical quasi-complementary series. In particular, as $G$ is non-compact, the trivial irreducible representation is not tempered and arises as the end of spherical complementary series $\triv = (\pi_{\triv, d}, \Ucal(\triv, d)) \in \Ghat$. Note that for $\sigma \in \MhatSD$ and $s = d - \ell(\sigma) \in \partial\Ical_\sigma$, the intertwining operator further descends to an operator on $\Ucal(\sigma, s)$ which, again by abuse of notation, we also denote by $\mathcal{A}(\sigma, s)$.
	
	In light of \cref{eqn: LtwoKsigma as induced representation}, we deduce the following elementary but crucial corollary of \cref{thm: non-tempered irreps}.
	
	\begin{cor}
		\label{cor: complementary series embedding}
		We have the following as unitary representations of $G$:
		\begin{enumerate}
			\item $(\pi_{\sigma, s}, \Ucal(\sigma, s))$ is contained in $(U^s, L^2(K : \sigma))$ for all $\sigma \in \MhatSD$ and $s \in \interior\Ical_\sigma$;
			\item $(\pi_{\sigma, s}, \Ucal(\sigma, s))$ is contained in $(U^s, L^2(K : \sigma)/\ker\mathcal{A}(\sigma, s))$ for all $\sigma \in \MhatSD$ and $s = d - \ell(\sigma) \in \partial\Ical_\sigma$.
		\end{enumerate}
	\end{cor}
	
	\begin{remark}
		We warn the reader that for the ends of complementary series $(\pi_{\sigma, s}, \Ucal(\sigma, s))$ for $\sigma \in \MhatSD$ and $s = d - \ell(\sigma) \in \partial\Ical_\sigma$, we do not have the stronger form: $(\pi_{\sigma, s}, \Ucal(\sigma, s))$ is contained in $(U^s, L^2(K : \sigma))$. Nevertheless, the weaker form suffices for our purposes; see the lifting argument in the proof of \cref{cor:matrix coeff expand}.
	\end{remark}
	
	As a consequence of Frobenius reciprocity (cf. \cite[Chapter 5, \S 5.5.1]{WarnerI}), for any $\sigma \in \Mhat$ and $s \in \C$, the multiplicity of a $K$-type $\tau \in \widehat{K}$ occurring in $\Ind_P^G(\sigma \otimes \chi_s \otimes \triv)$ is exactly the multiplicity of $\sigma$ occurring in $\tau$. Therefore, $\tau$ is contained in $\Ind_P^G(\sigma \otimes \chi_s \otimes \triv)$ if and only if $\sigma$ is contained in $\tau$, and in which case both multiplicities are $1$. Now, for $\sigma \in \MhatSD$, we apply \cref{thm: non-tempered irreps} to conclude the following. For all $s \in \interior\Ical_\sigma$, the previous characterization applies for the $K$-types in the complementary series $\mathcal{U}(\sigma, s)$. For $s = d - \ell(\sigma) \in \partial\Ical_\sigma$, further care is needed: a $K$-type $\tau \in \widehat{K}$ is contained in the end of complementary series $\mathcal{U}(\sigma, s)$ if and only if $\sigma$ is contained in $\tau$ and $\tau$ is not contained in $\ker \mathcal{A}(\sigma, s)$, and in which case the multiplicity is $1$. This is not a simple criteria; however, we have a result of Vogan quoted in \cref{thm: minimal K-type in quasi-complementary series} which will be useful. For this, we need to introduce the concept of minimal $K$-types as defined in loc. cit.; for convenience, we write it explicitly for our setting where $K \cong \SO(d + 1)$ (cf. \cite[\S 4]{Thieleker}).
	
	\begin{definition}[Minimal $K$-Type]
		A $K$-type $\tau = (\tau_1, \dotsc, \tau_{\lceil d/2\rceil}) \in \widehat{K}$ of a unitary representation $(\pi, \Hcal)$ of $G$ is called a \emph{minimal $K$-type} of $(\pi, \Hcal)$ if
		\begin{align}
			\label{eqn: minimal K-type}
			\lambda_\tau := \sum_{j = 1}^{\lceil d/2\rceil} \left(\tau_j + \frac{d + 1 - 2j}{2}\right)^2
		\end{align}
		is minimized among all the $K$-types contained in $(\pi, \Hcal)$.
	\end{definition}
	
	\begin{remark}
		In light of the formula for the (intrinsic) Casimir eigenvalue \cite[Chapter 5, \S 4, Proposition 5.28]{Knapp-LieGroups}, the above definition is equivalent to the condition that the (intrinsic) Casimir eigenvalue is minimized among all the $K$-types contained in $(\pi, \Hcal)$.
	\end{remark}
	
	The following is the result of Vogan from \cite[Chapter XV, \S 3, Theorem 15.10]{Knapp}.
	
	\begin{thm} 
		\label{thm: minimal K-type in quasi-complementary series}
		Let $\sigma \in \MhatSD$ and $s \in \Ical_\sigma$. All minimal $K$-types of $\Ind_P^G(\sigma \otimes \chi_{s - d/2} \otimes \triv)$ are contained in $\Ucal(\sigma, s)$.
	\end{thm}
	
	Recall the decomposition $\Ghat = \Ghat_{\temp} \sqcup \Ghat_{\qcomp}$. We further have the identification
	\begin{align*}
		\Ghat_{\qcomp} \cong \{(\sigma, s): \sigma \in \MhatSD, s \in \Ical_\sigma\}
	\end{align*}
	as topological spaces, where $\Ghat$ is equipped with the Fell topology and the right hand side is equipped with the standard topology.

	Let $(\pi,\Hcal)$ be a unitary representation of $G$. Using the above decomposition, together with the direct integral decomposition from \cref{subsec: Unitary representation theory: generalities}, we obtain a corresponding decomposition of the measure $m_{\qcomp}$ into measures $\{m_\sigma\}_{\sigma \in \MhatSD}$ such that
	$$\supp(m_\sigma)\subset \Ical_\sigma \qquad \text{for all $\sigma \in \MhatSD$},$$
	and
	\begin{align}\label{eq:definition of mu_upsilon}
		(\pi,\Hcal) = \int_{\Ghat}^\oplus (\pi_\xi, \Hcal_\xi) \, dm(\xi) = \int_{\Ghat_{\temp}}^\oplus (\pi_\xi, \Hcal_\xi) \,dm_{\temp}(\xi) \oplus \Hilbertoplus_{\sigma \in \MhatSD} \int_{\Ical_\sigma}^\oplus (\pi_{\sigma, s}, \Hcal(\sigma, s)) \, dm_\sigma(s)
	\end{align}
	where each $(\pi_\xi, \Hcal_\xi)$ is isotypic of type $\xi  \in \Ghat$.
	
	By definition of the direct integral, we obtain the following explicit decomposition for any vector $\phi \in \Hcal$: there exists a measurable section $\{\phi_\xi\}_{\xi \in \Ghat} = \{\phi_\xi\}_{\xi \in \Ghat_{\temp}} \sqcup \{\phi_{\sigma, s}\}_{\sigma \in \MhatSD, s \in \Ical_\sigma}$ with $\phi_\xi \in \Hcal_\xi$ and $\phi_{\sigma, s} \in \Hcal(\sigma, s)$, such that
	\begin{align}
		\label{eqn: measurable sections for direct integral}
		\phi = \int_{\Ghat}^\oplus \phi_\xi \, dm(\xi) = \int_{\Ghat_{\temp}}^\oplus \phi_\xi \, dm_{\temp}(\xi) + \sum_{\sigma \in \MhatSD} \int_{\Ical_\sigma}^\oplus \phi_{\sigma, s} \, dm_\sigma(s).
	\end{align}

	Finally, we record the following useful lemma on density of smooth compactly supported functions of a fixed $K$-type.
	Recall the projection operators which we obtain for any $L \in \{K, M\}$ and denote them by $\mathsf{P}_\tau$ for any $\tau \in \widehat{K}$ and $\mathsf{P}_\sigma$ for any $\sigma \in \Mhat$. We omit indicating the ambient representation as it is understood from context---in particular, we always view $L^2(\GmodGamma)$ as a right regular representation and $L^2(K)$ as a left regular representation.
	
	\begin{lem}
		\label{lem: compact dense in tau type}
		Let $\t \in \widehat{K}$ be a $K$-type. The subspace $\compactsmooth(\GmodGamma) \cap L^2(\GmodGamma)_\t $ is dense in $L^2(\GmodGamma)_\t$.
	\end{lem}
	
	\begin{proof}
		Let $\t \in \widehat{K}$ be a $K$-type, $\widetilde{\phi} \in L^2(\GmodGamma)_\t \subset L^2(\GmodGamma)$, and $\epsilon > 0$. Since $\compactsmooth(\GmodGamma) \subset L^2(\GmodGamma)$ is dense, there exists $\phi \in \compactsmooth(\GmodGamma)$ such that $\bigl\|\phi - \widetilde{\phi}\bigr\|_{L^2(\GmodGamma)} < \epsilon$. We now apply $\mathsf{P}_\t$. Firstly, since $\widetilde{\phi} \in L^2(\GmodGamma)_\t$, we have $\mathsf{P}_\t\widetilde{\phi} = \widetilde{\phi}$. Secondly, it follows from the definition of $\mathsf{P}_\t$ in \cref{eq:definition of P_tau} that $\mathsf{P}_\t\phi$ remains as a compactly supported smooth function. Therefore, applying $\mathsf{P}_\t$ gives $\mathsf{P}_\t\phi \in \compactsmooth(\GmodGamma) \cap L^2(\GmodGamma)_\t$ with
		\begin{align*}
			\bigl\|\mathsf{P}_\t \phi - \widetilde{\phi}\bigr\|_{L^2(\GmodGamma)} \leq \bigl\|\phi - \widetilde{\phi}\bigr\|_{L^2(\GmodGamma)} < \epsilon
		\end{align*}
		as desired.
	\end{proof}

	\section{Asymptotic expansion of matrix coefficients in terms of the Harish-Chandra C-function and Laplace transforms}
	\label{sec: Rep theoretic asymptotic expansion of matrix coefficients}
	
	In this section, we recount a purely representation theoretic asymptotic expansion of matrix coefficients of Edwards--Oh \cite{EdwardsOh}. We then use it to prove \cref{prop:Laplace transform rep theory side}, which is a general result on analytic continuations of Laplace transforms of matrix coefficients. Roughly speaking, combined with the results of the next section, \cref{prop:Laplace transform rep theory side} characterizes weak containment of quasi-complementary series within a given representation of $G$ in terms of failure of holomorphic extension of the Laplace transform to certain subsets of the negative real line corresponding to the set of parameters of the quasi-complementary series in question.
	
	\subsection{Intertwining operators, Eisenstein integrals, and the Harish-Chandra C-function}
	We define constants related to the intertwining operator introduced previously, and define other important operators.

	Let $\sigma \in \MhatSD$ and $s > d/2$. From the intertwining property in \cref{eqn: A intertwining property}, we know that $\mathcal{A}(\sigma, s)$ intertwines with $\lambda_K$. Consequently, by Schur's lemma, $\mathcal{A}(\sigma, s)$ acts as a scalar on each $K$-type of $L^2(K : \sigma)$. By the characterization of the $K$-types which occur in $\Ind_P^G(\sigma \otimes \chi_{s - d/2} \otimes \triv)$, we may write
	\begin{align*}
		\mathcal{A}(\sigma, s) = \sum_{\tau \in \widehat{K}: \tau \supset \sigma} a(\sigma, s, \tau) \mathsf{P}_\tau
	\end{align*}
	as operators on $L^2(K : \sigma)$, for some set of scalars $\{a(\sigma, s, \tau)\}_{\tau \in \widehat{K}: \tau \supset \sigma} \subset \C$. Positive semi-definiteness of the inner product $\langle\cdot, \cdot\rangle_{(\sigma, s)}$ implies that for all $\tau \in \widehat{K}$ containing $\sigma$, the scalars are real and non-negative. 
	We also deduce from the intertwining property that $(\lambda_K, \ker\mathcal{A}(\sigma, s))$ is a unitary representation of $K$. Note that using $\ker\mathcal{A}(\sigma, s) = \Zcal(\sigma, s)$, for all $K$-types $\tau \in \widehat{K}$ containing $\sigma$, we have
	\begin{align*}
		a(\sigma, s, \tau) = 0 \iff \tau \subset (\lambda_K, \ker\mathcal{A}(\sigma, s)).
	\end{align*}

	Let $\sigma \in \Mhat$. For each $K$-type $\tau \in \widehat{K}$ of $L^2(K : \sigma)$, we use any orthonormal basis $\{v_j\}_{j = 1}^{\dim(\sigma)\dim(\tau)} \subset L^2(K : \sigma)_\tau$ to define a corresponding special vector $\omega_\tau \in L^2(K : \sigma)_\tau$ defined by
	\begin{align*}
		\omega_\tau = \sum_{j = 1}^{\dim(\sigma)\dim(\tau)} \overline{v_j(e)}v_j.
	\end{align*}
	Then, for each vector $v \in L^2(K : \sigma)_\tau$, we have the identity
	\begin{align*}
		v(k) = \langle v, \lambda_K(k)\omega_\tau \rangle_{L^2(K)} \qquad \text{for all $k \in K$}.
	\end{align*}
	For any pair of $K$-types $\t_1, \t_2 \in \widehat{K}$, define an operator $\mathsf{T}_{\t_1}^{\t_2}: L^2(K : \sigma)_{\t_1} \to L^2(K : \sigma)_{\t_2}$ by
	\begin{align*}
		\mathsf{T}_{\t_1}^{\t_2}v = \int_M v(m) \lambda_K(m) \omega_{\t_2} \, dm \qquad \text{for all $v \in L^2(K : \sigma)_{\t_1}$}.
	\end{align*}
	These are called \emph{Eisenstein integrals}. Note that $\mathsf{T}_{\t_1}^{\t_2}$ intertwines with $\lambda_M$. In particular, for $\tau_1 = \tau_2 = \tau$, we have by Schur's lemma that $\mathsf{T}_\t^\t$ preserves $M$-types. Then, for any $\sigma \in \MhatSD$ and $s \in \Ical_\sigma$, the operator $\mathsf{T}_\t^\t$ preserves $\ker\mathcal{A}(\sigma, s)$ and hence descends to an operator on $\Ucal(\sigma, s)_\t$. We also have the following lemma.
	
	\begin{lem}[{\cite[Corollary 3.9]{EdwardsOh}}]
		\label{lem:T_tau as projection}
		Let $\sigma \in \Mhat$. Let $\tau \in \widehat{K}$ be a $K$-type of $L^2(K : \sigma)$. Then, we have
		\begin{align*}
			\mathsf{T}_\tau^\tau|_{L^2(K : \sigma)_\t} = \frac{\dim(\tau)}{\dim(\sigma)} \mathsf{P}_{\sigma^*}|_{L^2(K : \sigma)_\t}.
		\end{align*}
		When $s \in \Ical_\sigma$, the same holds on $\Ucal(\sigma, s)$.
	\end{lem}

	\begin{definition}[Harish-Chandra C-Function]
		\label{def:HC C function}
		Let $\sigma \in \Mhat$. For all $s > d/2$, define a bounded operator $\mathsf{C}_+(s): L^2(K : \sigma) \r L^2(K : \sigma)$, by
		\begin{align*}
			\mathsf{C}_+(s) = \int_{\overline{N}} e^{-s H(\overline{n})} U^s(\k(\overline{n})^{-1}) \, d\overline{n}.
		\end{align*}
		The bounded operator-valued function given by $s\mapsto \mathsf{C}_+(s)$ is called the \emph{Harish-Chandra C-function}.
	\end{definition}
	
	Indeed, $\mathsf{C}_+(s)$ is a well-defined bounded operator for all $s > d/2$ since $U^s$ is unitary and
	\begin{align*}
		\int_{\overline{N}} e^{-s H(\overline{n})} \, d\overline{n} < +\infty
	\end{align*}
	by \cite[Corollary of Lemma 45]{Harish-Chandra} (see \cite[pp. 303]{Harish-Chandra} for an explicit formula in terms of the $\Gamma$ function). The Harish-Chandra C-function $\mathsf{C}_+(s)$ intertwines with $U^s|_M = \lambda_M$ (see the proof of \cite[Lemma 4.2]{EdwardsOh}). When $s \in \Ical_\sigma$, it descends as above to a bounded operator on $\Ucal(\sigma,s)$. By Schur's lemma we may record a slightly refined version of \cite[Lemma 4.2]{EdwardsOh}.
	
	\begin{lem}
		\label{lem:C-function scalar on M-types}
		Let $\sigma \in \Mhat$ and $s > d/2$. The Harish-Chandra C-function $\mathsf{C}_+(s)$ acts as a scalar on each irreducible $M$-submodule of $L^2(K : \sigma)$. In particular, they preserve the $M$- and $K$-types of $L^2(K : \sigma)$. Consequently, when $s \in \Ical_\sigma$, the same holds on $\Ucal(\sigma, s)$.
	\end{lem}
	
	\subsection{Asymptotic expansion of matrix coefficients following Edwards--Oh}

	For $s \in (d/2, d]$, we set
	\begin{align}\label{eq:eta_s}
		\eta_s := \min\set{2s-d,1}>0.
	\end{align}
	
	We recall the following expansion of matrix coefficients due to Edwards and Oh.
	
	\begin{thm}[{\cite[Theorem 4.3]
			{EdwardsOh}}]
		\label{thm:matrix coeff expand}
		Let $I$ be a compact subset of $(d/2,d]$.
		Let $\sigma\in \Mhat$, $s \in \Ical_\sigma \cap I$, and $\t_1,\t_2\in \widehat{K}$ be $K$-types. For all $u\in L^2(K : \sigma)_{\t_1}$, $v \in L^2(K : \sigma)_{\t_2}$, and $t\geq 0$, we have
		\begin{align*}
			\langle \pi_{\sigma, s}(a_t) u, v \rangle_{L^2(K)}
			= e^{-(d-s)t} \langle \mathsf{T}_{\t_1}^{\t_2} \mathsf{C}_+(s) u,v \rangle_{L^2(K)}
			+ O_I \left(e^{-(d-s+\eta_s)t} \norm{\mathsf{T}_{\t_1}^{\t_2}}_{L^2(K)} \norm{u}_{L^2(K)} \norm{v}_{\Sob_K^1(L^2(K))}\right).
		\end{align*}
	\end{thm}
	
	\begin{remark}
		Although the original statement and proof of \cite[Theorem 4.3]
		{EdwardsOh} is for $\Ucal(\sigma, s)_{\t_1}$ and $\Ucal(\sigma, s)_{\t_2}$ for $\sigma \in \MhatSD$ and $s \in \interior\Ical_\sigma$, the proof for $L^2(K : \sigma)_{\t_1}$ and $L^2(K : \sigma)_{\t_2}$ goes through verbatim.
		
		Moreover, the uniformity of the implicit constant in loc. cit. is stated for compact subsets $I$ of the \textit{open} interval $(d/2,d)$; however, their proof shows that such uniformity continues to hold for compact subsets of the half-closed interval $(d/2,d]$.
	\end{remark}

	Specializing \cref{thm:matrix coeff expand} to the case $\tau_1=\tau_2=\tau$, we have the following corollary.

	\begin{cor} \label{cor:matrix coeff expand}
		Let $I$ be a compact subset of $(d/2,d]$.
		Let $\sigma\in \MhatSD$, and $s \in \Ical_\sigma \cap I$. Let $(\pi, \Hcal) = (\pi_{\sigma, s}, \Ucal(\sigma,s)) \widehat{\otimes} (\triv, \Mcal(\sigma,s))$ be a unitary representation of $G$, with $\Mcal(\sigma,s)$ a (possibly infinite dimensional) separable multiplicity space.
		Then, for any $K$-type $\t\in \widehat{K}$ containing $\sigma$, for all $\phi,\psi\in \Hcal_{\t}$ and $t\geq 0$, we have
		\begin{align}\label{eq: matrix coeffs }
			\langle \pi(a_t) \phi, \psi \rangle_{\Hcal}
			= e^{-(d-s)t} \langle \mathsf{T}_{\t}^{\t} \mathsf{C}_+(s) \phi,\psi \rangle_{\Hcal}
			+ O_{I,\tau} \left(e^{-(d-s+\eta_s)t}  \norm{\phi}_{\Hcal} \norm{\psi}_{\Hcal}\right).
		\end{align}
	\end{cor}
	
	\begin{proof}
		Let $I$, $\sigma$, $s$, and $\tau$ be as in the lemma. We first obtain the desired result for the unitary representation $(\pi_{\sigma, s}, \Ucal(\sigma,s))$ of $G$, i.e., assuming $\Mcal(\sigma,s)$ is a $1$-dimensional multiplicity space. First, we treat the complementary series, i.e., the case $s \in \interior\Ical_\sigma$. This case is simpler since we may fix an embedding $(\pi_{\sigma, s}, \Ucal(\sigma, s)) \subset (U^s, L^2(K : \sigma)) \subset (U^s, L^2(K))$ by \cref{cor: complementary series embedding}. From \cref{eq:intertwining}, for any $u \in \Ucal(\sigma,s)_\t$ and $v \in \Ucal(\sigma,s)$, we have
		$$\langle u,v\rangle_{\Ucal(\sigma,s)}=a(\sigma,s,\t)\langle u,v\rangle_{L^2(K)}.$$
		Letting $u, v \in \Ucal(\sigma,s)_{\t}$, we can use \cref{thm:matrix coeff expand} and then the above identity to obtain
		\begin{align*}
			&\langle \pi_{\sigma,s}(a_t) u, v \rangle_{\Ucal(\sigma,s)}\\
			={}&a(\sigma,s,\tau) \langle \pi_{\sigma,s}(a_t) u, v \rangle_{L^2(K)}\\
			={}&e^{-(d-s)t} a(\sigma,s,\tau) \langle \mathsf{T}_{\t}^{\t} \mathsf{C}_+(s) u,v \rangle_{L^2(K)}
			+ O_{I} \left(e^{-(d-s+\eta_s)t} a(\sigma,s,\tau) \norm{\mathsf{T}_{\t}^{\t}}_{L^2(K)} \norm{u}_{L^2(K)} \norm{v}_{\Sob_K^1(L^2(K))}\right)\\
			={}&e^{-(d-s)t} \langle \mathsf{T}_{\t}^{\t} \mathsf{C}_+(s) u,v \rangle_{\Ucal(\sigma,s)}
			+ O_{I} \left(e^{-(d-s+\eta_s)t} a(\sigma,s,\tau) \norm{\mathsf{T}_{\t}^{\t}}_{L^2(K)} \norm{u}_{L^2(K)} \norm{v}_{\Sob_K^1(L^2(K))}\right).
		\end{align*}
		By \cref{lem: Sobolev bound}, there exists a constant $C(\t) > 0$ such that
		$\norm{v}_{\Sob_K^1(L^2(K))}\leq C(\t)\norm{v}_{L^2(K)}$.
		Since $\norm{\mathsf{T}_{\t}^{\t}}_{L^2(K)}$ is $O_\t (1)$ and $a(\sigma,s,\tau)\|v\|_{L^2(K)}^2=\|v\|^2_{\Ucal(\sigma,s)}$, we have the bound
		$$ a(\sigma,s,\tau) \norm{\mathsf{T}_{\t}^{\t}}_{L^2(K)} \norm{u}_{L^2(K)} \norm{v}_{\Sob_K^1(L^2(K))}
		\ll_\t \|u\|_{\Ucal(\sigma,s)}\|v\|_{\Ucal(\sigma,s)}.$$
		Thus, we get
		\begin{align}
			\label{eqn:matrix coeff expand mult 1}
			\langle \pi_{\sigma,s}(a_t) u, v \rangle_{\Ucal(\sigma,s)}
			&= e^{-(d-s)t} \langle \mathsf{T}_{\t}^{\t} \mathsf{C}_+(s) u,v \rangle_{\Ucal(\sigma,s)}
			+ O_{I,\t} \left(e^{-(d-s+\eta_s)t}   \norm{u}_{\Ucal(\sigma,s)} \norm{v}_{\Ucal(\sigma,s)}\right).
		\end{align}
		Now, we treat the ends of complementary series, i.e., the case $s = d - \ell(\sigma) \in \partial\Ical_\sigma$. In this case, although we do not have the same embedding as in the first case, we have $(\pi_{\sigma, s}, \Ucal(\sigma, s)) \subset (U^s, L^2(K : \sigma)/\ker\mathcal{A}(\sigma, s))$ from \cref{cor: complementary series embedding}. For any given $u, v\in \Ucal(\sigma,s)_{\t}$, we may take \emph{arbitrary} lifts $\tilde{u}, \tilde{v} \in L^2(K : \sigma)_{\t} \subset L^2(K)$ and carry out the same argument as above with $u$, $v$, $\langle \cdot, \cdot \rangle_{\Ucal(\sigma, s)}$, and $\| \cdot \|_{\Ucal(\sigma, s)}$, replaced with $\tilde{u}$, $\tilde{v}$, $\langle \cdot, \cdot \rangle_{(\sigma, s)}$, and $\| \cdot \|_{(\sigma, s)}$, respectively, to obtain the analog of \cref{eqn:matrix coeff expand mult 1} for $\tilde{u}$ and $\tilde{v}$. Finally, observe that \cref{eqn:matrix coeff expand mult 1} then holds for $u$ and $v$ since the inner products and norms descend to $\Ucal(\sigma, s)$ as we have $(U^s, \ker\mathcal{A}(\sigma, s)) = (U^s, \Zcal(\sigma, s))$ and is a subrepresentation of $(U^s, L^2(K : \sigma))$.

		Now, consider the more general unitary representation $(\pi, \Hcal) = (\pi_{\sigma, s}, \Ucal(\sigma,s)) \widehat{\otimes} (\triv, \Mcal(\sigma,s))$ of $G$, with $\Mcal(\sigma,s)$ a separable multiplicity space. Let $\{v_j\}_{j=1}^n$ for some $n \in \Z_{\geq 0} \sqcup \{\infty\}$ be a countable orthonormal Hilbert basis for $\Mcal(\sigma,s)$. Let $\phi, \psi \in \Hcal_\t$. They have decompositions $\phi=\sum_{i=1}^n \phi_i\otimes v_i$ and $\psi=\sum_{j=1}^n \psi_j\otimes v_j$, with $\phi_i, \psi_j\in \Ucal(\sigma,s)_\t$ and $\|\phi\|_\Hcal^2=\sum_{i=1}^n \|\phi_i\|_{\Ucal(\sigma,s)}^2 < +\infty$ and $\|\psi\|_\Hcal^2=\sum_{j=1}^n \|\psi_j\|_{\Ucal(\sigma,s)}^2 < +\infty$. Since $\{v_j\}_{j=1}^n$ is orthonormal and $G$ acts tivially on $\Mcal(\sigma,s)$ we have
		\begin{align*}
			\langle \pi(a_t)\phi, \psi \rangle_\Hcal
			&= \left\langle\sum_{i=1}^n \pi_{\sigma, s}(a_t)\phi_i \otimes v_i, \sum_{j=1}^n \psi_j \otimes v_j\right\rangle_\Hcal \\
			&= \sum_{i, j=1}^n \langle \pi_{\sigma, s}(a_t)\phi_i, \psi_j\rangle_{\Ucal(\sigma,s)} \cdot \langle v_i, v_j\rangle_\Mcal \\
			&=\sum_{j=1}^n \langle \pi_{\sigma, s}(a_t) \phi_j, \psi_j \rangle_{\Ucal(\sigma,s)}.
		\end{align*}
		The corollary now follows by applying \cref{eqn:matrix coeff expand mult 1} to each summand.
	\end{proof}

	The following proposition is essentially known in the literature, though it is difficult to pinpoint a reference dealing with the full quasi-complementary series. We provide a proof for the sake of completeness.

	\begin{prop}\label{prop:MatrixDecay}
		Let $s_\star\in [d/2, d)$.
		Let $(\pi, \Hcal)$ be a unitary representation of $G$ which does not weakly contain any quasi-complementary series $\Ucal(\sigma,s)$ with $\sigma \in \MhatSD$ and $s \in \Ical_\sigma \cap (s_\star, d]$. Then, for all $\epsilon>0$ and $K$-finite vectors $u,v\in \Hcal$, we have the following bound on matrix coefficients:
		\begin{align*}
			\langle \pi(a_t) u, v\rangle_\Hcal \ll_G (1 + t)e^{-(d - s_\star)t} \bigl(\dim \langle \pi(K)u\rangle \dim \langle \pi(K)v\rangle\bigr)^{1/2} \|u\|_{\Hcal} \|v\|_{\Hcal} \qquad \text{for all $t > 0$}.
		\end{align*}
	\end{prop}
	
	\begin{proof}
		In the proof, we will encounter the spherical functions $\varphi_s: G \to \R$ for any $s \in \R$ defined by
		\begin{align*}
			\varphi_s(g) = \int_K e^{-s H(g^{-1}k)} \, dk \qquad \text{for all $g \in G$}.
		\end{align*}
		They coincide with the matrix coefficient of $(U^s, L^2(K : \triv))$ corresponding to $\mathbf{1}_K \in L^2(K : \triv)$ (or corresponding to any $K$-invariant unit vector in $\Ucal(\triv, s)$; see \cite[Lemma 5]{Harish-Chandra}).
		The particular spherical function $\Xi := \varphi_{d/2}$ is called the \emph{Harish-Chandra function}. We refer the reader to \cite[Chapter VII, \S 8]{Knapp} and \cite[Chapter 4, \S 4.6]{GangolliVaradarajan} for further details on spherical functions. In particular, we shall use their bounds which are originally due to Harish-Chandra \cite[Theorem 3]{Harish-Chandra}.
		
		Let $s_\star$ and $(\pi, \Hcal)$ be as in the proposition. By definition, we have the orthogonal decomposition $\Hcal = \Hcal_{\temp} \oplus \Hcal_{\qcomp}$ into the tempered part $\Hcal_{\temp}$ and the non-tempered part $\Hcal_{\qcomp}$. We treat each part separately.

		We first treat the tempered part. Let $u, v \in \Hcal_{\temp}$ be $K$-finite vectors. Then, by a result of Cowling--Haagerup--Howe \cite[Corollary]{CowlingHaagerupHowe} and bounds on the Harish-Chandra function, we have
		\begin{align}
			|\langle \pi(a_t)u, v\rangle_{\Hcal}| &\leq \Xi(a_t) \bigl(\dim\langle\pi(K)u\rangle \dim\langle\pi(K)v\rangle\bigr)^{1/2} \|u\|_{\Hcal} \|v\|_{\Hcal} \\
			\label{eqn: matrix coeff bound temp}
			&\ll_G (1 + t)e^{-(d/2)t} \bigl(\dim\langle\pi(K)u\rangle \dim\langle\pi(K)v\rangle\bigr)^{1/2} \|u\|_{\Hcal} \|v\|_{\Hcal} \qquad \text{for all $t > 0$}.
		\end{align}

		Now we treat the non-tempered part. By the classification result quoted in \cref{thm: non-tempered irreps}, only the quasi-complementary series can be weakly contained in $\Hcal_{\qcomp}$. We first establish \cref{eqn: matrix coeff bound comp} for all quasi-complementary series $\Ucal(\sigma, s)$ with parameters $\sigma \in \MhatSD$ and $s \in \Ical_\sigma \smallsetminus (s_\star, d]$. As the ends of complementary series can be handled using a lifting argument as in the proof of \cref{cor:matrix coeff expand}, we restrict our attention to the complementary series. To this end, let $\sigma \in \MhatSD$ and $s \in \interior\Ical_\sigma \smallsetminus (s_\star, d]$, and let $u, v \in \Ucal(\sigma, s)$ be $K$-finite vectors. By \cite[Chapter VII, \S 8, Proposition 7.14]{Knapp}, we have
		\begin{align*}
			\bigl|\langle \pi(a_t)u, v\rangle_{L^2(K)}\bigr| \ll_{u, v} \varphi_s(a_t).
		\end{align*}
		Define the subset
		\begin{align*}
			\mathcal{T} := \bigl\{\t \in \widehat{K}: \t \supset \sigma, \mathsf{P}_\tau(u) \neq 0, \mathsf{P}_\tau(v) \neq 0\bigr\} \subset \widehat{K}.
		\end{align*}
		By $K$-finiteness of $u$ and $v$, the subset $\mathcal{T}$ is finite. Converting to the unitary inner product $\langle \cdot, \cdot \rangle_{\Ucal(\sigma, s)}$ on $\Ucal(\sigma, s)$ and again using bounds on spherical functions, we have
		\begin{align}
			\bigl|\langle \pi_{\sigma, s}(a_t)u, v\rangle_{\Ucal(\sigma, s)}\bigr| &\ll_{u, v} \#\mathcal{T} \cdot \max_{\tau \in \mathcal{T}} a(\sigma, s, \t) \cdot \varphi_s(a_t) \\
			\label{eqn: matrix coeff bound comp}
			&\ll_{u, v, s} (1 + t)e^{-(d - s)t} \qquad \text{for all $t > 0$}.
		\end{align}
		
		Combining \cref{eqn: matrix coeff bound comp} with the fact from \cite[Chapter I, \S 5, Theorem 5.8]{Helgason} that the Haar measure on $G$ is given by
		\begin{align*}
			dg = \mu_M(M)^{-1} \sinh(t)^d \, dk \, dt \, dk
		\end{align*}
		and
		using the asymptotics
		$\sinh(t)^d\asymp e^{dt}$ for $t\gg 1$,
		we conclude that $\Ucal(\sigma, s)$ is strongly $L^p$ for $p := \frac{d}{d - s_\star}$ (in the terminology of Shalom \cite{Shalom}) for any $\sigma \in \MhatSD$ and $s \in \Ical_\sigma \smallsetminus (s_\star, d]$, i.e., the matrix coefficient $G \to \C$ given by $g \mapsto \langle \pi(g)u, v\rangle_{\Hcal}$ is in $L^{p + \epsilon}(G)$ for all $\epsilon > 0$.
		Since this holds true for all elements of $\Ghat$ weakly contained in $\Hcal_{\qcomp}$, the claimed bound of the proposition on matrix coefficients of $\Hcal_{\qcomp}$ follows by
		the equivalence of condition~2 and condition~4 in \cite[Theorem 2.1]{Shalom} and the fact that $d/p=d-s_\star$. Combined with the bounds on matrix coefficients of $\Hcal_{\temp}$ in \cref{eqn: matrix coeff bound temp}, this completes the proof of the proposition.
	\end{proof}

	We will now use the gathered tools to prove the following main proposition of this section.

	\begin{prop}
		\label{prop:Laplace transform rep theory side}
		Let $\d_\Hcal\in (d/2,d]$ and let $(\pi,\Hcal)$ be a unitary representation of $G$ which does not weakly contain any quasi-complementary series $\Ucal(\s,s)$ with parameters $\s\in \MhatSD$ and $s\in \Ical_\s\cap (\d_\Hcal,d]$.

		Let $\t\in \widehat{K}$ be a $K$-type.
		For all $\phi,\psi\in \Hcal_\t$, define the Laplace transform of the real and imaginary parts of the scaled matrix coefficient $F_\spadesuit: \{z \in \C: \Re(z) > 0\} \to \C$ by
		\begin{align*}
			F_\spadesuit(z) = \int_0^{+\infty} e^{-(z+\d_\Hcal - d)t} \spadesuit(\langle \pi(a_t)\phi , \psi \rangle_\Hcal) \,dt \qquad \text{for all $z\in \C$ with $\Re(z)>0$},
		\end{align*}
		for $\spadesuit \in \{\Re, \Im\}$. Let  $\eta_0=\min\{\d_\Hcal-d/2, 1\}$. Then, both $F_\Re$ and $F_\Im$ admit holomorphic extensions to
		$\set{z \in \C: \Re(z)>-\eta_0} \smallsetminus (-\eta_0,0]$.
		Moreover, on the half-plane $\{z \in \C: \Re(z)>-\eta\}$ for any fixed $\eta < \eta_0$, we may write
		\begin{align*}
			F_\spadesuit=A_\spadesuit+B_\spadesuit
		\end{align*}
		where $A_\spadesuit$ is a holomorphic function and $B_\spadesuit$ is a meromorphic function defined by
		\begin{align*}
			B_\spadesuit(z) = \sum_{\sigma \in \MhatSD: \sigma \subset \t} \int_{[\d_\Hcal-\eta,\d_\Hcal]}
			\frac{\spadesuit\bigl(\langle \mathsf{T}_\t^\t \mathsf{C}_+(s) \phi_{\sigma,s},\psi_{\sigma,s}\rangle_{\Hcal(\sigma,s)}\bigr)}{z + \d_\Hcal - s} \,dm_\sigma(s) \qquad \text{for all $z \in \C$ with $\Re(z) > -\eta$},
		\end{align*}
		for $\spadesuit \in \{\Re, \Im\}$.
		Moreover, $A_\Re$ and $A_\Im$ are uniformly bounded on the half-plane $\{z \in \C: \Re(z)\geq -\eta_1\}$ for any $\eta_1<\eta$.
	\end{prop}
	
	\begin{proof}
		Assume the hypotheses of the proposition. We give the proof for $F_\Re$ as the proof for $F_\Im$ is identical.
		Recall the direct integral decomposition and the spectral measures $\{m_\sigma\}_{\sigma \in \MhatSD}$ given in \cref{eq:definition of mu_upsilon}.
		For any $\eta\in (0,\d_\Hcal-d/2)$, define $I_\eta:=[\d_\Hcal-\eta,\d_\Hcal]$ and the Hilbert subspace
		$$\Hcal[\eta]=\Hilbertoplus_{\sigma\in \MhatSD}\int_{I_\eta}^{\oplus} \Hcal(\sigma,s) \, dm_\sigma(s) \qquad \subset \Hcal.$$
		Then, $\Hcal=\Hcal[\eta]\oplus\Hcal[\eta]^\bot$ is an orthogonal decomposition into subrepresentations, and for any $\varphi\in \Hcal$, denote by $\varphi_\eta \in \Hcal[\eta]$ and $\varphi_\eta^\bot \in \Hcal[\eta]^\bot$ its projections to $\Hcal[\eta]$ and $\Hcal[\eta]^\bot$, respectively. By orthogonality of these spaces, we have
		$$\langle \pi(a_t)\phi , \psi \rangle_{\Hcal}=\langle \pi(a_t) \phi_\eta , \psi_\eta \rangle_{\Hcal}+\bigl\langle \pi(a_t) \phi_\eta^\bot , \psi_\eta^\bot \bigr\rangle_{\Hcal} \qquad \text{for all $t \in \R$},$$
		and taking real parts, we have
		$$\Re(\langle \pi(a_t)\phi , \psi \rangle_{\Hcal})=\Re(\langle \pi(a_t)\phi_\eta , \psi_\eta \rangle_{\Hcal})+\Re\bigl(\bigl\langle \pi(a_t)\phi_\eta^\bot , \psi_\eta^\bot \bigr\rangle_{\Hcal}\bigr) \qquad \text{for all $t \in \R$}.$$
		Using the above decomposition we get a corresponding decomposition $F_\Re(z)=F_{\Re,\eta}(z)+F_{\Re,\eta}^\bot(z)$ where
		\begin{align*}
			F_{\Re,\eta}(z) = \int_0^{+\infty} e^{-(z+\d_\Hcal- d)t} \Re(\langle \pi(a_t)\phi_\eta , \psi_\eta \rangle_{\Hcal}) \,dt
		\end{align*}
		and similarly for $F^\bot_{\Re,\eta}(z)$. By \cref{prop:MatrixDecay} we have the bound
		$$\bigl|\Re\bigl(\bigl\langle \pi(a_t)\phi_\eta^\bot , \psi_\eta^\bot \bigr\rangle_{\Hcal}\bigr)\bigr| \leq \bigl|\bigl\langle \pi(a_t)\phi_\eta^\bot , \psi_\eta^\bot \bigr\rangle_{\Hcal}\bigr|\ll_{\phi, \psi} (t+1)e^{-(d-\d_\Hcal+\eta)t} \qquad \text{for all $t > 0$}.$$
		Thus, for $\Re(z)>-\eta$, the integral defining $F^\bot_{\Re,\eta}(z)$ converges absolutely:
		$$\bigl|F^\bot_{\Re,\eta}(z)\bigr| \leq \int_0^{+\infty} \bigl|e^{-(z+\d_\Hcal- d)t}  \Re\bigl(\bigl\langle \pi(a_t)\phi_\eta^\bot , \psi_\eta^\bot \bigr\rangle_{\Hcal}\bigr)\bigr| \,dt\ll \int_0^{+\infty} (t+1)e^{-(\Re(z)+\eta)t} \,dt < +\infty,$$
		and hence defines a holomorphic function that is uniformly bounded on $\{z \in \C: \Re(z)\geq -\eta_1\}$, for any $\eta_1 <\eta$.
		
		Next, we treat $F_{\Re,\eta}(z)$. Note that for any $\varphi \in \Hcal_\tau$ of the fixed $K$-type $\tau$, we have that for all $\sigma \in \MhatSD$ and $m_\sigma$-almost all $s\in I_\eta$, the component $\varphi_{\sigma,s}$ is also of $K$-type $\tau$. Since $\Ucal(\sigma,s)_\tau$ is non-trivial only if $\sigma\subset \tau$, we get
		$$\varphi_{\eta}=\sum_{\sigma \in \MhatSD: \sigma \subset \t}\int_{I_\eta}^\oplus \varphi_{\sigma,s} \, dm_\sigma(s) \qquad \in \Hcal[\eta]_\tau.$$
		Using such a decomposition for $\phi$ and $\psi$, we have
		$$\langle \pi(a_t)\phi_{\eta},\psi_{\eta}\rangle_\Hcal=\sum_{\sigma \in \MhatSD: \sigma \subset \t}\int_{I_\eta} \langle  \pi_{\sigma,s}(a_t)\phi_{\sigma,s}, \psi_{\sigma,s}\rangle_{\Hcal(\sigma,s)} \, dm_\sigma(s) \qquad \text{for all $t \in \R$}.$$
		For the compact interval $I_\eta=[\d_\Hcal-\eta,\d_\Hcal] \subset (d/2,d]$ and each of the finitely many $\sigma \in \MhatSD$ with $\sigma\subset \t$, we use \cref{cor:matrix coeff expand} to estimate for all $s \in I_\eta$ and $t \geq 0$ that
		\begin{align*}
			\langle \pi_{\sigma,s}(a_t) \phi_{\sigma,s}, \psi_{\sigma,s} \rangle_{\Hcal(\sigma,s)}
			= e^{-(d-s)t} \langle \mathsf{T}_{\t}^{\t} \mathsf{C}_+(s) \phi_{\sigma,s},\psi_{\sigma,s} \rangle_{\Hcal(\sigma,s)}
			+ O_{I_\eta,\tau} \left(e^{-(d-s+\eta_s)t}  \norm{\phi_{\sigma,s}}_{\Hcal(\sigma,s)} \norm{\psi_{\sigma,s}}_{\Hcal(\sigma,s)}\right).
		\end{align*}
		Hence, taking real parts and integrating, we get that
		\begin{align*}
			F_{\Re,\eta}(z)={}&\int_0^{+\infty} \sum_{\sigma \in \MhatSD: \sigma \subset \t}\int_{I_\eta} e^{-(z+\d_\Hcal-d)t} \Re\bigl(\langle \pi_{\sigma,s}(a_t) \phi_{\sigma,s}, \psi_{\sigma,s}\rangle_{\Hcal(\sigma,s)}\bigr) \, dm_\sigma(s) \, dt\\
			={}&\int_0^{+\infty} \sum_{\sigma \in \MhatSD: \sigma \subset \t}\int_{I_\eta} \bigg(e^{-(z+\d_\Hcal-s)t} \Re\bigl(\langle \mathsf{T}_{\t}^{\t} \mathsf{C}_+(s) \phi_{\sigma,s},\psi_{\sigma,s} \rangle_{\Hcal(\sigma,s)}\bigr)
			\\
			{}&+
			O_{I_\eta,\tau} \left(e^{-(z + \d_\Hcal - s + \eta_s) t}  \norm{\phi_{\sigma,s}}_{\Hcal(\sigma,s)} \norm{\psi_{\sigma,s}}_{\Hcal(\sigma,s)}\right) \bigg) \, dm_\sigma(s) \, dt.
		\end{align*}
		Write
		$$F_{\Re,\eta}(z)=B_\Re(z)+E_\Re(z)$$
		where $B_\Re(z)$ is the main term and $E_\Re(z)$ is the error term. Assume $\Re(z) > 0$ so that the integrals converge absolutely and we may interchange the order of integration. Then, we get
		\begin{align*}
			B_\Re(z) = \sum_{\sigma \in \MhatSD: \sigma \subset \t} \int_{I_\eta}
			\frac{\Re\bigl(\langle \mathsf{T}_\t^\t \mathsf{C}_+(s) \phi_{\sigma,s},\psi_{\sigma,s}\rangle_{\Hcal(\sigma,s)}\bigr)}{z + \d_\Hcal - s} \,dm_\sigma(s).
		\end{align*}
		Note that $B_\Re(z)$ is holomorphic on $\C\smallsetminus [-\eta,0]$.
		Moreover, the integral defining $E_\Re(z)$ is absolutely convergent on $\{z \in \C: \Re(z)> -\eta\}$ as we now show.
		Decomposing $I_\eta=(I_\eta\smallsetminus ( (d+1)/2,d]) \sqcup (I_\eta \smallsetminus (d/2, (d+1)/2])$, and using the definition of $\eta_s$ in \cref{eq:eta_s}, we have
		\begin{align*}
			|E_\Re(z)|
			\ll{}&_{I_\eta,\t} \sum_{\sigma \in \MhatSD: \sigma\subset\t}\int_{I_\eta}\|\phi_{\sigma,s}\|_{\Hcal(\sigma,s)} \|\psi_{\sigma,s}\|_{\Hcal(\sigma,s)}\int_0^{+\infty} e^{-(\Re(z)+\d_\Hcal-s+\eta_s)t} \, dt \, dm_\sigma(s)\\
			={}& \sum_{\sigma \in \MhatSD: \sigma\subset\t} \int_{I_\eta\smallsetminus ( (d+1)/2,d]}
			\|\phi_{\sigma,s}\|_{\Hcal(\sigma,s)} \|\psi_{\sigma,s}\|_{\Hcal(\sigma,s)}\int_0^{+\infty} e^{-(\Re(z)+\d_\Hcal+s-d)t}\, dt \, dm_\sigma(s)\\
			&{}+\sum_{\sigma \in \MhatSD: \sigma\subset\t} \int_{I_\eta \smallsetminus (d/2, (d+1)/2]}\|\phi_{\sigma,s}\|_{\Hcal(\sigma,s)} \|\psi_{\sigma,s}\|_{\Hcal(\sigma,s)}\int_0^{+\infty} e^{-(\Re(z)+\d_\Hcal-s+1)t}\, dt \, dm_\sigma(s) \\
			\leq{}& \sum_{\sigma \in \MhatSD: \sigma\subset\t} \int_{I_\eta\smallsetminus ( (d+1)/2,d]}
			\|\phi_{\sigma,s}\|_{\Hcal(\sigma,s)} \|\psi_{\sigma,s}\|_{\Hcal(\sigma,s)}\int_0^{+\infty} e^{-(\Re(z)+2\d_\Hcal-d-\eta)t}\, dt \, dm_\sigma(s)\\
			&{}+\sum_{\sigma \in \MhatSD: \sigma\subset\t} \int_{I_\eta \smallsetminus (d/2, (d+1)/2]}\|\phi_{\sigma,s}\|_{\Hcal(\sigma,s)} \|\psi_{\sigma,s}\|_{\Hcal(\sigma,s)}\int_0^{+\infty} e^{-(\Re(z)+1)t}\, dt \, dm_\sigma(s)\\
			\leq{}& \left(\frac{1}{\Re(z)+2\d_\Hcal-d-\eta)}+\frac{1}{\Re(z)+1}\right) \sum_{\sigma \in \MhatSD: \s \subset\t} \int_{I_\eta} \|\phi_{\sigma,s}\|_{\Hcal(\sigma,s)} \|\psi_{\sigma,s}\|_{\Hcal(\sigma,s)} \, dm_\sigma(s)\\
			\leq{}& \left(\frac{1}{\Re(z)+\d_\Hcal-d/2)}+\frac{1}{\Re(z)+1}\right) \|\phi_\eta\|_{\Hcal[\eta]} \|\psi_\eta\|_{\Hcal[\eta]}.
		\end{align*}
		It follows that $E_\Re(z)$ is bounded and holomorphic on $\{z \in \C: \Re(z)>-\min\{\eta, 1\}\}$.
		The result now follows with $A_\Re(z)=F_{\Re,\eta}^\bot(z)+E_\Re(z)$.
	\end{proof}
	
	\section{Non-vanishing of main terms of matrix coefficients}
	\label{sec:nonvanishing}
	In the work of Edwards--Oh \cite{EdwardsOh}, a key point is the \emph{vanishing} of the main terms for non-trivial $\sigma \in \MhatSD$ and $M$-invariant vectors $u\in \Ucal(\sigma,s)_{\t_1}$ in the asymptotic expansion of matrix coefficients, \cref{thm:matrix coeff expand}---for their purpose of eventually extracting exponential mixing of the geodesic flow with respect to the Bowen--Margulis--Sullivan measure. On the other hand, for our purpose of obtaining a strong spectral gap, we require the \emph{non-vanishing} of the main terms for non-trivial $\sigma \in \MhatSD$ for suitable choices of test vectors. This has the consequence that presence of a given quasi-complementary series can be detected using asymptotics of matrix coefficients. The goal of this section is to prove such a non-vanishing result.
	
	Let $s > d/2$. Recall from \cref{lem:C-function scalar on M-types} that the Harish-Chandra C-function $\mathsf{C}_+(s)$ act as a scalar on each irreducible $M$-submodule of $L^2(K : \sigma)$. Denote by $\mathsf{C}_+(\tau : \sigma; s)$ the scalar that $\mathsf{C}_+(s)$ acts by on an $M$-type $\sigma \in \Mhat$ contained in a $K$-type $\tau \in \widehat{K}$ (which occurs with multiplicity $1$; see \cref{subsec: Unitary representation theory of SO}). Precisely these scalars were computed by Eguchi--Koizumi--Mamiuda in \cite{EguchiKoizumiMamiuda}. We quote their theorem below.
	
	\begin{thm}[{\cite[Theorem 8.2]{EguchiKoizumiMamiuda}}]
		\label{thm:EguchiKoizumiMamiuda formula}
		Let $d \geq 2$ and $s > d/2$. Let $\sigma = (\sigma_1, \dotsc, \sigma_{\lfloor d/2\rfloor}) \in \Mhat$ be an $M$-type and $\tau = (\tau_1, \dotsc, \tau_{\lceil d/2\rceil}) \in \widehat{K}$ be a $K$-type such that $\sigma$ is contained in $\tau$. We have the following:
		\begin{enumerate}
			\item if $2 \mid d$, then
			\begin{align*}
				\mathsf{C}_+(\tau : \sigma; s) = \frac{(d - 1)!}{(d/2 - 1)!} \cdot \frac{\prod_{j = 1}^{d/2} \Gamma(s - d/2 + j - \sigma_j) \prod_{j = 1}^{d/2} \Gamma(s + d/2 - j + \sigma_j)}{\prod_{j = 1}^{d/2} \Gamma(s - d/2 + j - \tau_j) \prod_{j = 1}^{d/2} \Gamma(s + d/2 - j + 1 + \tau_j)};
			\end{align*}
			\item if $2 \nmid d$, then
			\begin{align*}
				\mathsf{C}_+(\tau : \sigma; s) = (\tfrac{d - 1}{2})! \, 2^{-2s + d} \, \Gamma(2s) \cdot \frac{\prod_{j = 1}^{(d - 1)/2} \Gamma(s - d/2 + j - \sigma_j) \prod_{j = 1}^{(d - 1)/2} \Gamma(s + d/2 - j + \sigma_j)}{\prod_{j = 1}^{(d + 1)/2} \Gamma(s - d/2 + j - \tau_j) \prod_{j = 1}^{(d + 1)/2} \Gamma(s + d/2 - j + 1 + \tau_j)}.
			\end{align*}
		\end{enumerate}
	\end{thm}
	
	We use \cref{thm:EguchiKoizumiMamiuda formula} to prove the following proposition. Recall the constant $\lambda_\tau$ associated to $K$-types $\tau \in \widehat{K}$ from \cref{eqn: minimal K-type}.
	
	\begin{prop}
		\label{prop:nonvanishing}
		Let $\sigma \in \Mhat$ be an $M$-type. There exists a $K$-type $\tau \in \widehat{K}$ such that
		\begin{enumerate}
			\item\label{itm: tau contains sigma and dual} $\tau$ contains both $\sigma$ and its dual $\sigma^*$;
			\item\label{itm: tau is minimal K-type} $\tau$ minimizes $\lambda_\tau$ among the $K$-types satisfying Property~\cref{itm: tau contains sigma and dual};
			\item\label{itm: nonvanishing} $\mathsf{C}_+(\tau : \sigma^*; s) \neq 0$ for all $s > d/2$.
		\end{enumerate}
	\end{prop}
	
	\begin{proof}
		Let $\sigma \in \Mhat$ be an $M$-type. Proceeding by cases, we take an explicit $K$-type $\tau \in \widehat{K}$ and verify the claimed properties.
		
		\medskip
		\noindent
		\textit{Case 1: $2 \mid d$}. Let us write $\sigma = (\sigma_1, \dotsc, \sigma_{d/2})$ where
		\begin{align*}
			\sigma_1 \geq \sigma_2 \geq \dotsb \geq \sigma_{d/2 - 1} \geq |\sigma_{d/2}|.
		\end{align*}
		Let us also write $\sigma^* = (\sigma^*_1, \dotsc, \sigma^*_{d/2})$. In the current case, recall that
		\begin{itemize}
			\item if $4 \mid d$, then $\sigma^* = \sigma = (\sigma_1, \dotsc, \sigma_{d/2})$;
			\item otherwise, $\sigma^* = (\sigma_1, \dotsc, \sigma_{d/2 - 1}, -\sigma_{d/2})$.
		\end{itemize}
		Take $\tau = (\tau_1, \dotsc, \tau_{d/2 - 1}, \tau_{d/2}) := (\sigma_1, \dotsc, \sigma_{d/2 - 1}, |\sigma_{d/2}|)$. Then $\tau$ contains both $\sigma$ and $\sigma^*$ since they satisfy the interlacing property
		\begin{align*}
			\tau_1 \geq \sigma_1 \geq \tau_2 \geq \sigma_2 \geq \dotsb \geq \tau_{d/2 - 1} \geq \sigma_{d/2 - 1} \geq \tau_{d/2} \geq |\pm\sigma_{d/2}|.
		\end{align*}
		This proves Property~\cref{itm: tau contains sigma and dual}. Inspecting \cref{eqn: minimal K-type} term by term verifies Property~\cref{itm: tau is minimal K-type}.
		
		Let $s > d/2$. The formula provided by \cref{thm:EguchiKoizumiMamiuda formula} for $\sigma^*$ and the choice of $\tau$ gives
		\begin{align*}
			\mathsf{C}_+(\tau : \sigma^*; s) = \frac{(d - 1)!}{(d/2 - 1)!} \cdot \frac{\Gamma(s - \sigma_{d/2}^*) N_2}{\Gamma(s - |\sigma_{d/2}|) D_2}
		\end{align*}
		where all but the last factor in the first iterated products in the numerator and denominator canceled out, and $N_2 = \prod_{j = 1}^{d/2} \Gamma(s + d/2 - j + \sigma^*_j)$ and
		$D_2 = \prod_{j = 1}^{d/2} \Gamma(s + d/2 - j + 1 + \tau_j)$ denote the second iterated products in the numerator and denominator. To prove Property~\cref{itm: nonvanishing}, we need to verify that zeros in the numerator $\Gamma(s - \sigma^*_{d/2}) N_2$ and poles in the denominator $\Gamma(s - |\sigma_{d/2}|) D_2$ are canceled out. The first condition holds simply because the $\Gamma$ function does not have zeros. Now we verify the second condition. Recall that the poles of $\Gamma$ are all simple and occur exactly at the non-positive integers. The factor $D_2$ does not have poles since
		\begin{align}
			\label{eqn: large Gamma argument}
			s + d/2 - j + 1 + \tau_j \geq s > 0 \qquad \text{for all $1 \leq j \leq d/2$}.
		\end{align}
		So we are left to examine the first factor $\Gamma(s - |\sigma_{d/2}|)$ in the denominator. Fortunately, this must cancel with one of the factors $\Gamma(s - \sigma_{d/2})$ or $\Gamma(s + \sigma_{d/2})$ in the numerator using the observation that
		\begin{align*}
			\Gamma(s - \sigma^*_{d/2})\Gamma(s + \sigma^*_{d/2}) = \Gamma(s - \sigma_{d/2})\Gamma(s + \sigma_{d/2}).
		\end{align*}
		Therefore, $\mathsf{C}_+(\tau : \sigma^*; s) \neq 0$.
		
		\medskip
		\noindent
		\textit{Case 2: $2 \nmid d$}. Let us write $\sigma = (\sigma_1, \dotsc, \sigma_{(d - 1)/2})$ where
		\begin{align*}
			\sigma_1 \geq \sigma_2 \geq \dotsb \geq \sigma_{(d - 1)/2} \geq 0.
		\end{align*}
		In the current case, $\sigma^* = \sigma$. Take $\tau = (\tau_1, \dotsc, \tau_{(d - 1)/2}, \tau_{(d + 1)/2}) := (\sigma_1, \dotsc, \sigma_{(d - 1)/2}, 0)$. Then $\tau$ contains $\sigma$ (and hence also $\sigma^*$) since they satisfy the interlacing property
		\begin{align*}
			\tau_1 \geq \sigma_1 \geq \tau_2 \geq \sigma_2 \geq \dotsb \geq \tau_{(d - 1)/2} \geq \sigma_{(d - 1)/2} \geq |\tau_{(d + 1)/2}|.
		\end{align*}
		This proves Property~\cref{itm: tau contains sigma and dual}. Inspecting \cref{eqn: minimal K-type} term by term verifies Property~\cref{itm: tau is minimal K-type}.
		
		Let $s > d/2$. As in Case~1, the formula provided by \cref{thm:EguchiKoizumiMamiuda formula} for $\sigma^*$ and the choice of $\tau$ gives
		\begin{align*}
			\mathsf{C}_+(\tau : \sigma^*; s) = (\tfrac{d - 1}{2})! \, 2^{-2s + d} \,  \cdot \frac{\Gamma(2s) N_2}{\Gamma(s + 1/2) D_2}
		\end{align*}
		and to prove Property~\cref{itm: nonvanishing}, we only need to verify that the poles in the denominator $\Gamma(s + 1/2) D_2$ are canceled out. Again, $D_2$ does not have poles by the same inequality \cref{eqn: large Gamma argument} for all $1 \leq j \leq (d + 1)/2$.
		Similarly, $\Gamma(s + 1/2)$ does not have poles since $s + 1/2 > 0$. Therefore, $\mathsf{C}_+(\tau : \sigma^*; s) \neq 0$.
	\end{proof}

	The following corollary shows that the main terms in the asymptotic expansion provided by \cref{thm:matrix coeff expand} are non-vanishing for a suitable choice of test vectors.
	
	\begin{cor}\label{cor:nonvanishing}
		Let $\sigma \in \MhatSD$ be a self-dual $M$-type. Let $\t \in \widehat{K}$ be the $K$-type provided by \cref{prop:nonvanishing}. Then, for all $s \in \Ical_\sigma$, the subspace $\Ucal(\sigma,s)_\t\cap \Ucal(\sigma,s)_{\sigma} \subset \Ucal(\sigma,s)$ is non-trivial and the operator $\mathsf{T}_\t^\t \mathsf{C}_+(s)$ acts on it by a non-zero scalar.
	\end{cor}
	
	\begin{proof}
		Let $\sigma$ and $\tau$ be as in the corollary. By \cref{prop:nonvanishing}\cref{itm: tau contains sigma and dual}, $\tau$ contains $\sigma$, whence $\tau$ is contained in $\Ind_P^G(\sigma \otimes \chi_{s - d/2} \otimes \triv)$ (see \cref{subsec: Unitary representation theory of G}). By \cref{prop:nonvanishing}\cref{itm: tau is minimal K-type}, $\tau$ is a minimal $K$-type of $\Ind_P^G(\sigma \otimes \chi_{s - d/2} \otimes \triv)$ and so by \cref{thm: minimal K-type in quasi-complementary series}, we further deduce that the subspace $\Ucal(\sigma,s)_\tau$ is non-trivial. Again since $\tau$ contains $\sigma$, we conclude that $\Ucal(\sigma,s)_\t\cap \Ucal(\sigma,s)_{\sigma}$ is non-trivial. Recall that \cref{lem:T_tau as projection} gives $\mathsf{T}_\t^\t|_{\Ucal(\sigma, s)_\t} =\frac{\dim(\t)}{\dim(\sigma)} \mathsf{P}_{\sigma^*}|_{\Ucal(\sigma, s)_\t}$, and that $\mathsf{P}_{\sigma^\ast} = \mathsf{P}_{\sigma}$  acts as identity on $\Ucal(\sigma,s)_{\sigma}$.
		Moreover, we have by \cref{prop:nonvanishing}\cref{itm: nonvanishing} that $\mathsf{C}_+(s)$ acts by the scalar $\mathsf{C}_+(\tau : \sigma; s) \neq 0$ on $\Ucal(\sigma,s)_\t\cap \Ucal(\sigma,s)_{\sigma}$. Putting everything together, we have
		\begin{align*}
			[\mathsf{T}_\t^\t \mathsf{C}_+(s)]|_{\Ucal(\sigma,s)_\t \cap \Ucal(\sigma,s)_{\sigma}} = \frac{\dim(\t)}{\dim(\sigma)} \mathsf{C}_+(\t : \sigma; s)\cdot \Id_{\Ucal(\sigma,s)_\t \cap \Ucal(\sigma,s)_{\sigma}} \neq 0,
		\end{align*}
		as desired.
	\end{proof}
	
	\section{Decay of matrix coefficients with power saving error terms and Laplace transforms}
	\label{sec:exp mixing}
	
	In this section, we introduce the key dynamical input: the decay of matrix coefficients with power saving error terms which is a consequence of exponential mixing of the frame flow. We use this to obtain meromorphic extensions of Laplace transforms of the scaled matrix coefficients to a strip to the left of the imaginary axis with at most one simple pole at the origin.

	We quote the following theorem from \cite[Theorem 1.2]{SarkarWinter} and \cite[Theorem 1.2]{LiPanSarkar} in a simplified form.
	
	\begin{thm}[{\cite{SarkarWinter,LiPanSarkar}}]
		\label{thm:decay of correlations for Haar}
		There exist non-zero Borel measures $\BR$ and $\BRstar$, and $\eta > 0$ such that for all $\phi, \psi\in \compactsmooth(\GmodGamma)$ and $t > 0$, we have
		\begin{align*}
			\int_{\GmodGamma} (\phi \circ a_t) \psi \, d\Haar
			= e^{-(d-\critexp)t}\int_{\GmodGamma} \phi \, d\BR \int_{\GmodGamma} \psi \, d\BRstar
			+ O_{\phi,\psi}\bigl(e^{-(d-\critexp + \eta) t}\bigr),
		\end{align*}
		where the implicit constant depends on Sobolev norms of $\phi$ and $\psi$, as well as their supports.
	\end{thm}
	
	As alluded to previously, the above theorem is derived from \cite[Theorem 1.1]{SarkarWinter} and \cite[Theorem 1.1]{LiPanSarkar} regarding exponential mixing of the frame flow with respect to the Bowen--Margulis--Sullivan (BMS) measure $\BMS$ which is the Borel probability measure of maximal entropy.
	In the case $\critexp >\max\{d/2, d-1\}$, this result was obtained previously by Mohammadi--Oh in \cite{MohammadiOh}.
	Exponential mixing of the geodesic flow is due to Li--Pan \cite{LiPan}; see also \cite{Khalil-MixingII,Khalil-MixingI}. In the convex cocompact setting, exponential mixing of the frame flow is due to \cite{SarkarWinter}. The aforementioned derivation uses Roblin's transverse intersection argument \cite{Roblin}; see also \cite{OhShah,MohammadiOh,KelmerOh,EdwardsLeeOh} for more expositions of the argument, the latter being the most general. We also refer to loc. cit. for explicit formulas and further details for the BMS measure $\BMS$ and the Burger--Roblin (BR) measures $\BR$ and $\BRstar$ which appear in the above theorem---in this article, knowing their existence and the form of the main term in the above theorem suffices.

	\begin{remark}
		\cite[Theorem 1.2]{SarkarWinter} and \cite[Theorem 1.2]{LiPanSarkar} are stated, a priori, for torsion-free $\Gamma$. It also holds for non-torsion-free $\Gamma$ by passing to a torsion-free subgroup of finite index $\tilde{\Gamma} \leq \Gamma$ by Selberg's lemma; the measures $\Haar$, $\BMS$, $\BR$, and $\BRstar$ on $\tilde{\Gamma}\bsl G$ are simply the lifts of those on $\GmodGamma$ up to normalization by the constant factor $[\Gamma: \tilde{\Gamma}]$.
	\end{remark}
	
	\begin{cor}
		\label{cor:Laplace transform and exp mixing}
		Let $\eta > 0$ be as in \cref{thm:decay of correlations for Haar}.
		For all $\phi,\psi\in \compactsmooth(\GmodGamma)$, define the Laplace transform of the real and imaginary parts of the scaled matrix coefficient $F_\spadesuit: \{z \in \C: \Re(z) > d - \critexp\} \to \C$ by
		\begin{align}
			F_\spadesuit(z) = \int_0^{+\infty} e^{-(z+\critexp-d)t} \spadesuit\bigl(\langle \phi \circ a_t, \psi\rangle_{L^2(\GmodGamma)}\bigr) \, dt \qquad \text{for all $z \in \C$ with $\Re(z) > d - \critexp$},
		\end{align}
		for $\spadesuit \in \{\Re, \Im\}$. Then, $F_\spadesuit$ admits a meromorphic extension to the half-plane $\{z \in \C: \Re(z) > -\eta\}$ with at most a simple pole at $z=0$ of residue $\spadesuit\Bigl(\int_{\GmodGamma} \phi \, d\BR \int_{\GmodGamma} \overline{\psi} \, d\BRstar\Bigr)$, for $\spadesuit \in \{\Re, \Im\}$.
	\end{cor}
	
	\begin{proof}
		Let $\eta>0$ be a parameter satisfying \cref{thm:decay of correlations for Haar} and let $\phi,\psi\in \compactsmooth(\GmodGamma)$. We give the proof for $F_\Re$ as the proof for $F_\Im$ is identical.
		Denote $M_{\phi, \psi} := \Re\Bigl(\int_{\GmodGamma} \phi \, d\BR \int_{\GmodGamma} \overline{\psi} \, d\BRstar\Bigr)$.
		Then, by \cref{thm:decay of correlations for Haar}, the function
		\begin{align*}
			F_\Re(z) - \int_0^{+\infty} e^{-zt} M_{\phi, \psi} \, dt = F_\Re(z)-\frac{M_{\phi, \psi}}{z},
		\end{align*}
		defined initially on $\{z \in \C: \Re(z) > d - \critexp\}$, extends holomorphically to the half-plane $\{z \in \C: \Re(z) >-\eta\}$.
		The corollary now follows since $M_{\phi,\psi}/z$ extends meromorphically to $\C$ with at most a simple pole at $z = 0$ of residue $M_{\phi, \psi}$.
	\end{proof}

	\section{Comparison of Laplace transforms and the proofs of Theorems~\ref{thm:SSG conjecture} and \ref{thm:upgraded decay of correlations for Haar}}
	\label{sec:Laplace comparison}
	
	In this section, we prove \cref{thm:SSG conjecture}.
	Property \cref{itm: gap for complementary series} of \cref{def:SSG} is established in \cref{prop:SSG conjecture part 2}, while Property \cref{itm: non-spherical complementary series at critexp} is proved in \cref{prop:SSG conjecture part 1}.
	An outline of the proof strategy is given in \cref{sec:outline}.

	\subsection{Preliminaries on Stieltjes transforms}
	
	Recall that for a signed measure $\nu$, infinite values $+\infty$ and $-\infty$ are permitted but not both. Consequently, in the Jordan decomposition $\nu = \nu^+ - \nu^-$ into the positive and negative parts $\nu^\pm$ of $\nu$ (which are positive measures), at least one of $\nu^+$ or $\nu^-$ must be finite. Moreover, if $\nu$ is finite, then both $\nu^+$ and $\nu^-$ must be finite, and hence its total variation $|\nu| = \nu^+ + \nu^-$ is also finite. We refer the reader to \cite[Chapter 3]{Folland} for further details. We also record following elementary lemma.

	\begin{lem}
		\label{lem: 0 on intervals 0 measure}
		Let $I \subset \R$ be an open interval and $E \subset I$ be a dense subset. Let $\nu$ be a signed Borel measure on $I$ such that $\nu([a, b]) = 0$ for all closed intervals $[a, b] \subset I$ with $a, b \in E$. Then, $\nu$ is the $0$ measure.
	\end{lem}
	\begin{proof}
		Let $I$, $E$, and $\nu$ be as in the lemma. We have the Jordan decomposition $\nu = \nu^+ - \nu^-$ where the positive and negative parts $\nu^\pm$ are positive Borel measures on $I$. It follows by hypothesis that $\nu^+([a, b]) = \nu^-([a, b])$ for all intervals $[a, b] \subset I$ with $a, b \in E$. Since the set of closed intervals whose endpoints are in the dense subset $E \subset I$ generate the Borel $\s$-algebra on $I$, it follows that $\nu^+ = \nu^-$. Hence, the lemma follows.
	\end{proof}
	
	\begin{remark}
		We cannot repeat the above proof for closed intervals $I \subset \R$ since the dense subset $E \subset I$ may not contain either endpoints of $I$.
	\end{remark}

	The following is the Stieltjes inversion formula corresponding to a Stieltjes transform for \emph{signed} measures. It is proved in \cite[Chapter~XIII, \S 65]{Wall} for \emph{positive} measures; however, its elementary proof can be adapted for signed measures. We include it here for completeness, especially because the proposition plays a crucial role in the proof of our main theorem.

	\begin{prop}[Stieltjes Inversion Formula]
		\label{prop:Stieltjes inversion}
		Let $\nu$ be a finite signed Borel measure on $\R$.
		Let $F: \C \smallsetminus\R \to \C$ be the \textbf{Stieltjes transform} of $\nu$ defined by
		\begin{align*}
			F(z) = \int_{\R} \frac{1}{z - t} \, d\nu(t) \qquad \text{for all $z \in \C \smallsetminus\R$}.
		\end{align*}
		Then $F$ is holomorphic on $\C \smallsetminus\R$ and for any closed interval $[a, b] \subset \R$, we have
		\begin{align*}
			\frac{\nu([a, b)) + \nu((a, b])}{2} = -\frac{1}{\pi} \lim_{y \searrow 0} \int_a^b \Im F(x + iy) \, dx.
		\end{align*}
		In particular, if $\nu(\{a\}) = \nu(\{b\}) = 0$ and $F$ extends continuously to $(\C \smallsetminus\R) \cup [a,b]$, and hence holomorphically to $(\C \smallsetminus\R) \cup (a,b)$, then $\nu([a,b])=0$.
	\end{prop}
	
	\begin{proof}
		Let $\nu$, $F$, and $[a, b]$ be as in the proposition. Define a finite signed Borel measure $\tilde{\nu}$ on $[a, b] \times \R$ by $d\tilde{\nu}(x, t) = dx \, d\nu(t)$. For any compact subset $\mathcal{K} \subset \R_{>0}$ and $y \in \mathcal{K}$, we have the uniform bound
		\begin{align*}
			\int_{[a, b] \times \R} \left|\frac{1}{x + iy - t}\right| \, d|\tilde{\nu}|(x, t) \leq \frac{1}{\inf \mathcal{K}} \cdot |\tilde{\nu}|([a, b] \times \R) < +\infty.
		\end{align*}
		By similar uniform bounds, we may apply Fubini's theorem for integrals of $F$ over arbitrary line segments in $\C \smallsetminus \R$ which are parallel to the real or imaginary axis.
		Hence, holomorphicity of $z\mapsto \frac{1}{z - t}$ on $\C \smallsetminus \R$ together with Fubini's theorem imply vanishing of contour integrals of $F$ over rectangular contours in $\C\smallsetminus \R$ (with sides parallel to the real and imaginary axes).
		It follows that $F$ is holomorphic on $\C\smallsetminus \R$ by Morera's theorem (see \cite[Chapter 2, \S 5, Theorem 5.1]{SteinShakarchi} and \cite[Chapter 2, \S 7, Problem 3]{SteinShakarchi}).

		By the above bound, for any $y > 0$, we may integrate the imaginary part of $F$ over $[a, b]$ and apply Fubini's theorem again to obtain
		\begin{align*}
			\int_a^b \Im F(x + iy) \, dx &= \int_a^b \int_{\R} \frac{-y}{(x - t)^2 + y^2} \, d\nu(t) \, dx = \int_{\R} \int_a^b \frac{-1/y}{(x - t)^2/y^2 + 1} \, dx \, d\nu(t) \\
			&= \int_{\R} \arctan((t - b)/y) \, d\nu(t) - \int_{\R} \arctan((t - a)/y) \, d\nu(t).
		\end{align*}
		Now, we introduce the step function $h = -\chi_{(-\infty, 0)} + \chi_{(0, +\infty)}: \R \to \R$ and observe the pointwise limit
		\begin{align*}
			\lim_{y \searrow 0} \arctan(t/y) = \frac{\pi}{2}h(t) \qquad \text{for all $t \in \R$}.
		\end{align*}
		Recall that $|\arctan| \leq \frac{\pi}{2}$ and the total variation $|\nu|$ is finite since $\nu$ is finite. Thus, we may apply the Lebesgue dominated convergence theorem to obtain the limit
		\begin{align*}
			\lim_{y \searrow 0} \int_a^b \Im F(x + iy) \, dx &= \lim_{y \searrow 0} \int_{\R} \arctan((t - b)/y) \, d\nu(t) - \lim_{y \searrow 0} \int_{\R} \arctan((t - a)/y) \, d\nu(t) \\
			&= \frac{\pi}{2}\int_{\R} h(t - b) \, d\nu(t) - \frac{\pi}{2}\int_{\R} h(t - a) \, d\nu(t) \\
			&= -\frac{\pi}{2}\bigl(\nu((-\infty, b)) - \nu((b, +\infty)) - \nu((-\infty, a)) + \nu((a, +\infty))\bigr) \\
			&= -\pi\frac{\nu([a, b)) + \nu((a, b])}{2}.
		\end{align*}

		Let us deduce the last claim of the proposition. Suppose $F$ extends continuously to $(\C \smallsetminus \R) \cup [a, b]$. Observe that for any rectangular contour (with sides parallel to the real and imaginary axes) that intersects $(a, b)$, we may decompose it into 3 rectangular contours where only one of them intersects $(a, b)$ and is arbitrarily thin in the imaginary direction. In light of this observation, we again deduce using Morera's theorem (see loc. cit. in \cite{SteinShakarchi}) that $F$ is holomorphic on $(\C \smallsetminus \R) \cup (a, b)$. Now, since $\nu$ is a \emph{real}-valued measure, we have the property that $\overline{F(z)} = F(\overline{z})$ for all $z \in \C \smallsetminus \R$ and hence $\Im F(z) = 0$ for all $z \in [a, b]$ by continuity. Thus, by uniform continuity of $F$ on $[a, b] + i[-1, 1]$, we have
		\begin{align*}
			\lim_{y \searrow 0} \int_a^b \Im F(x + iy) \, dx = \int_a^b \lim_{y \searrow 0} \Im F(x + iy) \, dx = 0.
		\end{align*}
		Suppose further that $\nu(\{a\}) = \nu(\{b\}) = 0$. Then, it follows from the proven identities that $\nu([a, b]) = \frac{\nu([a, b)) + \nu((a, b])}{2} = 0$.
	\end{proof}

	\begin{cor}
		\label{cor: complex measure zero}
		Let $\nu=\nu_\Re+i\nu_\Im$ be a complex Borel measure on $\R$ where $\nu_\Re$ and $\nu_\Im$ are signed measures, and consider the functions $F_\Re$ and $F_\Im$ defined by
		\begin{align*}
			F_\Re(z)&=\int_\R \frac{1}{z-t} \, d\nu_\Re(t), & F_\Im(z)&=\int_\R \frac{1}{z-t} \, d\nu_\Im(t), \qquad \text{for all $z \in \C \smallsetminus \R$}.
		\end{align*}
		Suppose there exists a closed interval $[a, b] \subset \R$ such that the following holds:
		\begin{enumerate}
			\item\label{itm: no atoms on dense subset} there exists a dense subset $E\subset [a,b]$ such that $\nu(\{x\})=0$ for all $x\in E$;
			\item\label{itm: Stieltjes transforms extend continuously} $F_\Re$ and $F_\Im$ extend continuously to $(\C \smallsetminus\R) \cup [a,b]$.
		\end{enumerate}
		Then $\nu|_{(a,b)}$ is the zero measure.
	\end{cor}
	\begin{proof}
		Since $F_\Re$ and $F_\Im$ extend continuously to $[a,b]$, then, by \cref{prop:Stieltjes inversion}, we conclude that $\nu([\alpha,\beta])=0$ for any $\alpha,\beta\in E$ with $\alpha<\beta$. Hence, by \cref{lem: 0 on intervals 0 measure} we get that both $\nu_\Re$ and $\nu_\Im$ restricted to $(a,b)$ are the zero measures. The corollary follows since $\nu = \nu_\Re + i\nu_\Im$.
	\end{proof}

	\subsection{Proof of Theorem~\ref{thm:SSG conjecture}}
	
	We now prove the main theorem. We need the following fact in order to be able apply \cref{prop:Laplace transform rep theory side}.
	
	\begin{prop}
		\label{prop: no quasi-complementary series above critexp}
		For all $\upsilon \in \MhatSD$ and $s > \critexp$, the quasi-complementary series $\Ucal(\upsilon, s)$ is not weakly contained in $L^2(\GmodGamma)$.
	\end{prop}
	
	\begin{proof}

		By qualitative mixing of the frame flow together with Roblin's transverse intersection argument, cf.\ \cite[Theorem 1.4]{Winter}, we have for all $\phi,\psi\in \mathit{C}_{\mathrm{c}}(\GmodGamma)$, the estimate $\langle \phi \circ a_t,\psi\rangle_{L^2(\GmodGamma)} = O_{\phi, \psi}(e^{-(d-\critexp)t})$ as $t\to+\infty$.\footnote{This is the qualitative form of \cref{thm:decay of correlations for Haar} with no error rate.}
		Hence, the matrix coefficients $g\mapsto \langle \rho_G(g)\phi,\psi\rangle_{L^2(\GmodGamma)}$ belong to $L^{p+\epsilon}(G)$ for all $\epsilon>0$, and for $p:=\frac{d}{d - \critexp}$.
		This follows by a similar argument to the proof of \cref{prop:MatrixDecay} using the formula for the Haar measure on $G$.
		Since $\mathit{C}_{\mathrm{c}}(\GmodGamma)$ is dense in $L^2(\GmodGamma)$, this means that the representation $(\rho_G, L^2(\GmodGamma))$ is strongly $L^p$ by definition (in the terminology of \cite{Shalom}).
		Hence, \cite[Theorem 2.1]{Shalom} implies that every representation that is weakly contained in $L^2(\GmodGamma)$ is also strongly $L^p$.
		On the other hand, given $\upsilon \in \MhatSD$ and $s>\critexp$, the quasi-complementary series $\Ucal(\upsilon,s)$ is not strongly $L^p$, and hence cannot be weakly contained in $L^2(\GmodGamma)$.
		Indeed, by \cref{cor:nonvanishing}, there exists a $K$-type $\t \in \widehat{K}$ such that for all $s \in \Ical_\upsilon$, we have that the operator $\mathsf{T}_\t^\t\mathsf{C}_+(s)$ is non-zero on $\Ucal(\upsilon, s)_\t$. Combined with \cref{cor:matrix coeff expand}, this shows that there exist $K$-finite vectors in $\Ucal(\upsilon, s)$ for which the matrix coefficients are not in $L^{p+\epsilon}(G)$ for any sufficiently small $\epsilon>0$. Hence, \cite[Theorem 2.1]{Shalom} implies that the representation $\Ucal(\upsilon,s)$ cannot be strongly $L^p$.
	\end{proof}
	
	In what follows, we denote by $\eta_{\mrm{rep}} > 0$ and $\eta_{\mrm{mix}} > 0$ any constants for which \cref{prop:Laplace transform rep theory side} and \cref{cor:Laplace transform and exp mixing} hold, respectively.

	\begin{prop}
		\label{prop:SSG conjecture part 2}
		Fix $\eta := \min\{\eta_{\mrm{rep}}, \eta_{\mrm{mix}}\} \in (0,\critexp-d/2)$. For all $\upsilon \in \MhatSD$ and $s \in (\critexp - \eta, \critexp)$, the quasi-complementary series $\Ucal(\upsilon, s)$ is not weakly contained in $L^2(\GmodGamma)$.
	\end{prop}
	
	\begin{proof}
		
		Fix the constant $\eta := \min\{\eta_{\mrm{rep}}, \eta_{\mrm{mix}}\} \in (0,\critexp-d/2)$ as in the proposition. Let us denote $\LtwoGmodGamma := L^2(\GmodGamma)$.

		For the purposes of our proof, let us define the following objects for all $K$-types $\tau \in \widehat{K}$ and $M$-types $\sigma \in \MhatSD$. First define the Hilbert subspaces
		\begin{align*}
			\LtwoGmodGamma_{\t,\sigma} &:= \int_{\Ical_\sigma\cap (\critexp-\eta,\critexp)}^\oplus \LtwoGmodGamma(\sigma,s)_\t \,dm_\sigma(s) \qquad \subset \LtwoGmodGamma_\t \subset \LtwoGmodGamma, \\
			\Jcal_{\t,\sigma} &:=\int_{\Ical_\sigma \cap (\critexp-\eta,\critexp)}^\oplus  \LtwoGmodGamma(\sigma,s)_\t \cap \LtwoGmodGamma(\sigma,s)_{\sigma} \,dm_\sigma(s)
			\qquad \subset \LtwoGmodGamma_{\t,\sigma} \subset \LtwoGmodGamma_\t \subset \LtwoGmodGamma.
		\end{align*}
		Note that the Hilbert subspaces corresponding to different elements of $\MhatSD$ are mutually orthogonal, i.e.,
		\begin{align}\label{eq:orthogonal complementary series dir int}
			\sigma, \sigma' \in \MhatSD \text{ such that } \sigma \neq \sigma' \Longrightarrow
			\LtwoGmodGamma_{\t,\sigma} \perp \LtwoGmodGamma_{\t,\sigma'}.
		\end{align}
		Define the bounded operator $\mathsf{Q}_{\t,\sigma}: \LtwoGmodGamma_{\t,\sigma} \to \LtwoGmodGamma_{\t,\sigma}$ by
		\begin{align*}
			\mathsf{Q}_{\t,\sigma} := \int_{\Ical_\sigma \cap (\critexp-\eta,\critexp)}^\oplus \mathsf{T}_\t^\t|_{\LtwoGmodGamma(\sigma,s)_\t} \mathsf{C}_+(s)|_{\LtwoGmodGamma(\sigma,s)_\t}  \,dm_\sigma(s).
		\end{align*}
		We then extend it to an operator $\mathsf{Q}_{\t,\sigma}: \LtwoGmodGamma \to \LtwoGmodGamma$ by setting it to be the $0$ operator on $(\LtwoGmodGamma_{\t,\sigma})^\perp$. Finally, we define the operator $\mathsf{Q}_\t: \LtwoGmodGamma \to \LtwoGmodGamma$ by
		\begin{align}
			\mathsf{Q}_\t := \sum_{\sigma \in \MhatSD: \sigma \subset \t} \mathsf{Q}_{\t,\sigma}.
		\end{align}

		Now we begin the proof in earnest. Suppose for the sake of contradiction that $\Ucal(\upsilon, s_\star)$ is weakly contained in $\LtwoGmodGamma$ for some $\upsilon \in \MhatSD$ and $s_\star \in (\critexp - \eta, \critexp)$.
		Let $\t\in \widehat{K}$ be the corresponding $K$-type provided by \cref{prop:nonvanishing}. We first claim that 
		\begin{align}\label{eq:nonvanishing of main terms sum}
			\text{there exist functions $\phi, \psi \in \compactsmooth(\GmodGamma) \cap \LtwoGmodGamma_\t$ such that } \langle \mathsf{Q}_\t \phi, \psi \rangle_\LtwoGmodGamma \neq 0.
		\end{align}
		Indeed, by our assumption, $s_\star \in \supp(m_\upsilon)$ and hence $m_\upsilon((\critexp - \eta, \critexp)) > 0$. By \cref{cor:nonvanishing}, for all $s \in \Ical_\upsilon$, the Hilbert subspace $\LtwoGmodGamma(\upsilon,s)_\t\cap \LtwoGmodGamma(\upsilon,s)_{\upsilon} \subset \LtwoGmodGamma(\upsilon,s)$ is non-trivial and the operator $\mathsf{T}_\t^\t \mathsf{C}_+(s)$ acts on it by a non-zero scalar. It follows from the definition of direct integrals that:
		\begin{itemize}
			\item $\Jcal_{\t,\upsilon} \subset \LtwoGmodGamma_{\t,\upsilon}$ are \emph{non-trivial} Hilbert subspaces of $\LtwoGmodGamma_\t \subset \LtwoGmodGamma$;
			\item $\mathsf{Q}_{\t,\upsilon}$ acts as a \emph{non-zero} bounded operator on $\LtwoGmodGamma_{\t, \upsilon}$, $\LtwoGmodGamma_\t$, and $\LtwoGmodGamma$;
			\item $\mathsf{Q}_\t$ acts as a \emph{non-zero} bounded operator on $\LtwoGmodGamma_\t$ and $\LtwoGmodGamma$, due to the orthogonality relations in \cref{eq:orthogonal complementary series dir int}.
		\end{itemize}
		By the above facts, there exist $\widetilde{\phi}, \widetilde{\psi} \in \LtwoGmodGamma_\t$ such that
		\begin{align}\label{eq:Qcal_upsilon}
			\bigl\langle \mathsf{Q}_\t \widetilde{\phi}, \widetilde{\psi} \bigr\rangle_\LtwoGmodGamma \neq 0.
		\end{align}
		Since $\compactsmooth(\GmodGamma) \cap \LtwoGmodGamma_\t$ is a dense subspace of $\LtwoGmodGamma_\t$ by \cref{lem: compact dense in tau type}, Claim \eqref{eq:nonvanishing of main terms sum} follows.

		Next, we can rewrite \cref{eq:nonvanishing of main terms sum} more explicitly as follows:
		\begin{align*}
			\langle \mathsf{Q}_{\t} \phi, \psi\rangle_\LtwoGmodGamma
			= \sum_{\sigma \in \MhatSD: \sigma \subset \t}
			\int_{(\critexp-\eta,\critexp)} \mathbf{1}_{\Ical_{\sigma}}(s)
			\langle  \mathsf{T}_\t^\t \mathsf{C}_+(s) \phi_{\sigma,s}, \psi_{\sigma,s} \rangle_{\LtwoGmodGamma(\sigma,s)} \,d m_{\sigma}(s) \neq 0.
		\end{align*}
		This motivates introducing the complex measure $\nu_\t$ on $\R$ defined by
		\begin{align*}
			d\nu_\t(s) = \sum_{\sigma \in \MhatSD: \sigma\subset \t} \mathbf{1}_{\Ical_\sigma}(s) \langle  \mathsf{T}_\t^\t \mathsf{C}_+(s) \phi_{\sigma,s}, \psi_{\sigma,s} \rangle_{\LtwoGmodGamma(\sigma,s)} \, d m_\sigma(s).
		\end{align*}
		Then, we may further rewrite \cref{eq:nonvanishing of main terms sum} as
		\begin{align}
			\label{eq:nonvanishing of main terms measure}
			\nu_\t((\critexp - \eta, \critexp)) \neq 0.
		\end{align}
		By \cref{lem: m has countable atoms}, $\nu_\t$ has at most a countable number of atoms, which we recall are all singletons. As a consequence, we obtain a dense subset
		\begin{align*}
			E_\t = \{x \in \R: \text{$\{x\}$ is not an atom of $\nu_\t$}\} \subset \R
		\end{align*}
		with the property that $\nu_\t(\{x\}) = 0$ for all $x \in E_\t$. Thus, Property~\cref{itm: no atoms on dense subset} of \cref{cor: complex measure zero} is satisfied.
		
		Now, \cref{prop: no quasi-complementary series above critexp} says that $(\rho_G,\LtwoGmodGamma)$ does not weakly contain any quasi-complementary series $\Ucal(\s,s)$ with parameters $\s\in \MhatSD$ and $s>\critexp$. Thus, \cref{prop:Laplace transform rep theory side} applies to the representation $(\pi,\Hcal) = (\rho_G,\LtwoGmodGamma)$ with $\d_\Hcal=\critexp$. Let $\spadesuit \in \{\Re, \Im\}$. Applying the proposition, the Laplace transform of the $\spadesuit$ part of the scaled matrix coefficient defined by
		\begin{align}
			F_\spadesuit(z) = \int_0^{+\infty} e^{-(z+\critexp- d)t} \spadesuit\bigl(\langle \phi \circ a_t, \psi \rangle_{L^{2}(\GmodGamma)}\bigr) \,dt \qquad \text{for all $z \in \C$ with $\Re(z) > 0$}
		\end{align}
		admits a holomorphic extension to $\{z \in \C: \Re(z)>-\eta_{\mrm{rep}}\} \smallsetminus (-\eta_{\mrm{rep}}, 0]$ and can be written as a sum
		\begin{align*}
			F_\spadesuit = A_\spadesuit + B_\spadesuit
		\end{align*}
		where $A_\spadesuit$ is a bounded holomorphic function and $B_\spadesuit$ is a meromorphic function both on the half-plane $\{z \in \C: \Re(z) > -\eta_{\mrm{rep}}\}$ and $B_\spadesuit$ is defined explicitly by the formula
		\begin{align}
			B_\spadesuit(z)
			&= \sum_{\sigma \in \MhatSD: \sigma \subset \t} \int_{[\critexp-\eta,\critexp]}
			\frac{\spadesuit\bigl(\langle \mathsf{T}_\t^\t \mathsf{C}_+(s) \phi_{\sigma,s}, \psi_{\sigma,s} \rangle_{\LtwoGmodGamma(\sigma,s)}\bigr)}{z + \critexp - s} \, d m_{\sigma}(s) \qquad \text{for all $z \in \C$ with $\Re(z) > -\eta_{\mrm{rep}}$}.
		\end{align}
		Note that the integral in the definition of $B_\spadesuit$ is over the \emph{closed} interval $[\critexp-\eta,\critexp]$ as opposed to the integral in the definition of $\mathsf{Q}_{\t}$ which is over the \emph{open} interval $(\critexp-\eta,\critexp)$; only the right endpoint $\critexp$ is of significance here. In particular,  note that $s = \critexp$, a priori, may contribute a pole to $F_\spadesuit$; but in this proposition, we are only concerned with detecting whether the parameters $s \in (\critexp-\eta,\critexp)$ contribute poles to $F_\spadesuit$. \cref{prop:SSG conjecture part 1} below will treat the endpoint $s = \critexp$. 
		
		Write $\nu_\t|_{[\critexp-\eta,\critexp]} = \bigl(\nu_\t|_{[\critexp-\eta,\critexp]}\bigr)_\Re + i\bigl(\nu_\t|_{[\critexp-\eta,\critexp]}\bigr)_\Im$ where $\bigl(\nu_\t|_{[\critexp-\eta,\critexp]}\bigr)_\Re$ and $\bigl(\nu_\t|_{[\critexp-\eta,\critexp]}\bigr)_\Im$ are signed measures. Observe also that $B_\spadesuit(z)$ is the Stieltjes transform of $\bigl(\nu_\t|_{[\critexp-\eta,\critexp]}\bigr)_\spadesuit$ at $z+\critexp$.

		Now, \cref{cor:Laplace transform and exp mixing} also applies for our choice of $\phi$ and $\psi$. Let $\spadesuit \in \{\Re, \Im\}$. Applying the corollary, $F_\spadesuit$ admits a meromorphic extension to $\{z \in \C: \Re(z) > -\eta_{\mrm{mix}}\}$, with exactly one simple pole at $z=0$ of residue $\spadesuit\Bigl(\int_{\GmodGamma} \phi \, d\BR \int_{\GmodGamma} \overline{\psi} \, d\BRstar\Bigr)$.
		In particular, $B_\spadesuit$ extends holomorphically to intervals on the real line of the form $[a, b]$, for all $-\eta < a < b < 0$. Thus, Property~\cref{itm: Stieltjes transforms extend continuously} of \cref{cor: complex measure zero} is satisfied.
		We may now apply \cref{cor: complex measure zero} to conclude that $\nu_\t|_{(\critexp - \eta, \critexp)}$ is the zero measure.
		In particular, this contradicts \cref{eq:nonvanishing of main terms measure} and completes the proof.
	\end{proof}

	To prepare for the proof of the second proposition, we set some notation.
	Given a unitary representation $(\pi,\Hcal)$ of $G$, and more generally, of $M$, we write
	\begin{align*}
		\Hcal=\Hcal^M\oplus \Hcal^0
	\end{align*}
	for the orthogonal decomposition of $\Hcal$ into the subspace $\Hcal^M$ of $M$-invariant vectors and its orthogonal complement $\Hcal^0$. In particular, given a $K$-type $\t\in\widehat{K}$, we also write
	\begin{align*}
		\Hcal_\t=\Hcal_\t^M\oplus \Hcal_\t^0.
	\end{align*}
	For $v\in \Hcal$, we denote by $v^M$ and $v^0$ its orthogonal projections onto $\Hcal^M$ and $\Hcal^0$ respectively, i.e.,
	\begin{align}
		\label{eqn: M invariant projection}
		v^M &= \int_M \pi(m) v \, dm, & v^0 = v - v^M.
	\end{align}

	We record the following lemma, which is an immediate consequence of \cref{lem:T_tau as projection}.
	\begin{lem}
		\label{lem:T_tau preserves M-0 decomp}
		Let $\sigma\in \MhatSD$ and $s\in\Ical_\sigma$. Let $\t\in \widehat{K}$ be a $K$-type containing $\sigma$. Then, the operator $\mathsf{T}_\t^\t$ preserves the orthogonal decomposition $\Ucal(\sigma,s)_\t = \Ucal(\sigma,s)_\t^M \oplus \Ucal(\sigma,s)_\t^0$.
		Moreover, if $\sigma$ is non-trivial, then $\mathsf{T}_\t^\t \left|_{\Ucal(\sigma,s)_\t^M} \right.$ vanishes.
	\end{lem}
	\begin{proof}
		Let $\sigma\in \MhatSD$ and $s\in\Ical_\sigma$.
		The first assertion follows by recalling that $\mathsf{T}_\t^\t$ preserves $M$-types. \Cref{lem:T_tau as projection} asserts that the operator $\mathsf{T}_\t^\t$ restricted to $\Ucal(\sigma,s)_\t$ is a scalar multiple of the orthogonal projection onto $\Ucal(\sigma,s)_{\sigma}$ as $\sigma$ is self-dual. The second assertion follows upon noting that $\Ucal(\sigma,s)_\t^M$ is orthogonal to the subspace $\Ucal(\sigma,s)_{\sigma}$ since $\sigma$ is non-trivial.
	\end{proof}

	Recall from \cref{rem:SSG(1) Justification} that having established Property \cref{itm: gap for complementary series} of \cref{def:SSG} in \cref{prop:SSG conjecture part 2}, for any non-trivial $\upsilon \in \MhatSD$, a weak containment of $\Ucal(\upsilon, \critexp)$ in $L^2(\GmodGamma)$ is automatically a strong containment. We now prove the second proposition, establishing \cref{def:SSG}\cref{itm: non-spherical complementary series at critexp} and completing the proof of \cref{thm:SSG conjecture}.
	
	\begin{prop}
		\label{prop:SSG conjecture part 1}
		Let $\upsilon\in \MhatSD$ be such that $\critexp\in \Ical_\upsilon$. Then, we have the following.
		\begin{enumerate}
			\item\label{itm:simplicity of bottom of spectrum} If $\upsilon =\triv$ is trivial, then the quasi-complementary series $\Ucal(\triv,\critexp)$ is contained in $L^2(\GmodGamma)$ exactly with multiplicity $1$.
			\item\label{itm:non-sphericals do not occur at bottom} If $\upsilon$ is non-trivial, then the quasi-complementary series $\Ucal(\upsilon, \critexp)$ is not contained in $L^2(\GmodGamma)$.
		\end{enumerate}
	\end{prop}
	
	\begin{proof}
		Fix some $\upsilon\in \MhatSD$ such that $\critexp\in \Ical_\upsilon$. If $\critexp = d$, then $\upsilon = \triv$ and hence $\Ucal(\upsilon, \critexp) = \Ucal(\triv, d) = \triv \in \Ghat$, and $\Gamma < G$ is a lattice. Thus, Part~\cref{itm:simplicity of bottom of spectrum} trivially follows. We may now assume $\critexp < d$.
		If $\upsilon=\triv$, let $\t = \triv\in \widehat{K}$ be the trivial $K$-type.
		Otherwise, if $\upsilon$ is non-trivial, let $\t\in\widehat{K}$ be the $K$-type provided by \cref{prop:nonvanishing}.

		Let us denote $\LtwoGmodGamma = L^2(\GmodGamma)$. First, having established \cref{prop:SSG conjecture part 2}, we may isolate the contribution of $\LtwoGmodGamma(\sigma, \critexp)$ for $\sigma \in \MhatSD$ satisfying $\sigma \subset \t$ and $\Ucal(\sigma, \critexp) \subset \LtwoGmodGamma$ in the following fashion (cf. \cref{rem:SSG(1) Justification}). Fix $s_\star := \critexp - \eta$ where $\eta := \min\{\eta_{\mrm{rep}}, \eta_{\mrm{mix}}\}$ as in \cref{prop:SSG conjecture part 2} and define
		\begin{align*}
			\LtwoGmodGamma(s_\star) := \Hilbertoplus_{\sigma \in \MhatSD} \int_{\Ical_\sigma \cap (d/2, s_\star]}^\oplus \LtwoGmodGamma(\sigma, s) \, dm_\sigma(s).
		\end{align*}
		By \cref{prop:SSG conjecture part 2}, $(s_\star, \critexp) \subset \Ical_\sigma \smallsetminus \supp(m_\sigma)$ for all $\sigma \in \MhatSD$ and so using \cref{lem: m strong/weak containment}, we have the orthogonal decomposition
		\begin{align*}
			\LtwoGmodGamma &= \LtwoGmodGamma_{\temp} \oplus \Hilbertoplus_{\sigma \in \MhatSD} \int_{\Ical_\sigma}^\oplus \LtwoGmodGamma(\sigma, s) \, dm_\sigma(s) \\
			&= \LtwoGmodGamma_{\temp} \oplus \Hilbertoplus_{\sigma \in \MhatSD} \int_{\Ical_\sigma \cap (d/2, \critexp)}^\oplus \LtwoGmodGamma(\sigma, s) \, dm_\sigma(s) \oplus \Hilbertoplus_{\sigma \in \MhatSD: m_\sigma(\{\critexp\}) > 0} \LtwoGmodGamma(\sigma, \critexp) \\
			&= \LtwoGmodGamma_{\temp} \oplus \LtwoGmodGamma(s_\star) \oplus \Hilbertoplus_{\sigma \in \MhatSD: \Ucal(\sigma, \critexp) \subset \LtwoGmodGamma} \LtwoGmodGamma(\sigma, \critexp).
		\end{align*}

		Let $\phi, \psi \in \compactsmooth(\GmodGamma) \cap \LtwoGmodGamma_\t$; in particular, they are $K$-finite. Let $\phi_\star$ and $\psi_\star$ denote the orthogonal projections of $\phi$ and $\psi$ onto $\LtwoGmodGamma_\star := \LtwoGmodGamma_{\temp} \oplus \LtwoGmodGamma(s_\star)$. Applying \cref{prop:MatrixDecay} for $\LtwoGmodGamma_\star$, we have
		\begin{align}
			\label{eqn: matrix coeff bounds LtwoGmodGammastar}
			|\langle \phi_\star \circ a_t, \psi_\star \rangle_{\LtwoGmodGamma}| \ll_{\phi_\star, \psi_\star} (1 + t)e^{-(d - s_\star)t} \qquad \text{for all $t > 0$}.
		\end{align}
		Recall the direct integral notation from \cref{eqn: measurable sections for direct integral}.
		Applying \cref{cor:matrix coeff expand} for each orthogonal summand $\LtwoGmodGamma(\sigma, \critexp)$ above, we have
		\begin{align*}
			\langle \pi_{\sigma, \critexp}(a_t)\phi_{\sigma, \critexp}, \psi_{\sigma, \critexp}\rangle_{\LtwoGmodGamma(\sigma, \critexp)} ={}& e^{-(d-\critexp)t} \langle \mathsf{T}_{\t}^{\t} \mathsf{C}_+(\critexp) \phi_{\sigma, \critexp}, \psi_{\sigma, \critexp} \rangle_{\LtwoGmodGamma(\sigma, \critexp)} \\
			{}&+ O_\tau\left(e^{-(d-\critexp+\eta_{\critexp})t}  \norm{\phi_{\sigma, \critexp}}_{\LtwoGmodGamma(\sigma, \critexp)} \norm{\psi_{\sigma, \critexp}}_{\LtwoGmodGamma(\sigma, \critexp)}\right) \qquad \text{for all $t > 0$}.
		\end{align*}
		Note that for all $\sigma \in \MhatSD$, the orthogonal projections $\phi_{\sigma, \critexp}$ and $\psi_{\sigma, \critexp}$ are non-trivial only if $\sigma \subset \t$ and $\Ucal(\sigma, \critexp) \subset \LtwoGmodGamma$.  Thus, we have
		\begin{align}
			\label{eqn: matrix coeff critexp main term rep}
			e^{(d - \critexp)t} \langle \phi \circ a_t, \psi \rangle_{\LtwoGmodGamma} = \sum_{\sigma \in \MhatSD: \sigma \subset \t, \Ucal(\sigma, \critexp) \subset \LtwoGmodGamma} \langle \mathsf{T}_\t^\t \mathsf{C}_+(\critexp) \phi_{\sigma,\critexp},\psi_{\sigma,\critexp}\rangle_{\LtwoGmodGamma(\sigma,\critexp)} + O_{\phi, \psi}(e^{-\eta' t}) \qquad \text{for all $t > 0$}
		\end{align}
		where we choose any positive $\eta' < \min\{\eta, \eta_{\critexp}\}$.
		Here, we note that the sum on the right hand side of \cref{eqn: matrix coeff critexp main term rep} has finitely many terms
		since there are finitely many $\sigma \in \MhatSD$ with $\sigma \subset \t$.
		
		On the other hand, by \cref{thm:decay of correlations for Haar}, we have
		\begin{align}
			\label{eqn: matrix coeff critexp main term dyn}
			e^{(d - \critexp)t} \langle \phi \circ a_t, \psi \rangle_{\LtwoGmodGamma} = \int_{\GmodGamma} \phi \, d\BR \int_{\GmodGamma} \overline{\psi} \, d\BRstar + O_{\phi, \psi}(e^{-\eta_{\mrm{mix}} t}) \qquad \text{for all $t > 0$}.
		\end{align}

		In light of \cref{eqn: matrix coeff critexp main term rep,eqn: matrix coeff critexp main term dyn}, we respectively define sesquilinear forms $Q_{\mrm{rep}}$ on $\LtwoGmodGamma_\t$ and $Q_{\mrm{mix}}$ on $\compactsmooth(\GmodGamma)\cap \LtwoGmodGamma_\t$ as follows:
		\begin{align*}
			Q_{\mrm{rep}}(\phi,\psi) &:=
			\sum_{\sigma \in \MhatSD: \sigma \subset \t, \Ucal(\sigma, \critexp) \subset \LtwoGmodGamma}
			\langle \mathsf{T}_\t^\t \mathsf{C}_+(\critexp) \phi_{\sigma,\critexp},\psi_{\sigma,\critexp}\rangle_{\LtwoGmodGamma(\sigma,\critexp)}
			\qquad
			\text{for all $\phi,\psi \in \LtwoGmodGamma_\t$}, \\
			Q_{\mrm{mix}}(\phi,\psi) &:= \int_{\GmodGamma} \phi \, d\BR \int_{\GmodGamma} \overline{\psi} \, d\BRstar
			\qquad
			\text{for all $\phi,\psi \in \compactsmooth(\GmodGamma) \cap \LtwoGmodGamma_\t$}.
		\end{align*}
		Then, equating \cref{eqn: matrix coeff critexp main term rep,eqn: matrix coeff critexp main term dyn}, and taking $t \to +\infty$, we obtain
		\begin{align}\label{eq:Q_rep=Q_mix}
			Q_{\mrm{rep}}(\phi,\psi) = Q_{\mrm{mix}}(\phi,\psi) \qquad \text{for all $\phi,\psi \in \compactsmooth(\GmodGamma) \cap \LtwoGmodGamma_\t$}.
		\end{align}

		We now prove Part~\cref{itm:simplicity of bottom of spectrum}. In this case, $\t = \triv \in \widehat{K}$ is the trivial $K$-type and we have the following:
		\begin{itemize}
			\item $\sigma = \triv \in \MhatSD$ is the only $M$-type such that $\sigma \subset \t$;
			\item $\mathsf{T}_\t^\t$ acts by the identity operator on $\Lcal(\triv,\critexp)_\t$ by \cref{lem:T_tau as projection};
			\item $\mathsf{C}_+(\critexp)$ acts by a non-zero scalar $\mathsf{C}_+(\t: \triv; \critexp)$ on $\Lcal(\triv,\critexp)_\t$; see \cref{sec:nonvanishing}.
		\end{itemize}
		Hence, we have
		\begin{align}
			\label{eqn: Q_rep on critexp spherical complementary series}
			Q_{\mrm{rep}}(\phi,\psi) = \mathsf{C}_+(\t: \triv; \critexp) \cdot \langle  \phi_{\triv,\critexp},\psi_{\triv,\critexp}\rangle_{\LtwoGmodGamma(\triv,\critexp)} \qquad \text{for all $\phi,\psi \in \LtwoGmodGamma_\t$}.
		\end{align}
		Note that $\Ucal(\triv, \critexp)$ is contained in $\LtwoGmodGamma$ with multiplicity $\geq 1$.
		Indeed, otherwise, $Q_{\mrm{rep}}$ would be the zero sesquilinear form, contradicting \cref{eq:Q_rep=Q_mix} and the fact that $Q_{\mrm{mix}}$ is certainly not the zero sesquilinear form.

		Now, suppose that $\Ucal(\triv,\critexp)$ is contained in $\LtwoGmodGamma$ with multiplicity at least $2$.
		Then, we can find two orthogonal unit vectors $u, v \in \Lcal(\triv,\critexp)_\t$.
		Since $\compactsmooth(\GmodGamma)\cap \LtwoGmodGamma_\t$ is a dense subspace of $\LtwoGmodGamma_\t$ by \cref{lem: compact dense in tau type}, we can find sequences $\{\phi^{(n)}\}_{n \in \N}$ and $\{\psi^{(n)}\}_{n \in \N}$ in $\compactsmooth(\GmodGamma)\cap \LtwoGmodGamma_\t$ such that $\phi^{(n)}\r u$ and $\psi^{(n)}\r v$ in $\LtwoGmodGamma_\t$ as $n \to \infty$. In particular,
		\begin{align}\label{eq:vectors witnessing multiplicity}
			\bigl\langle \phi^{(n)}_{\triv,\critexp} , \phi^{(n)}_{\triv,\critexp} \bigr\rangle_{\LtwoGmodGamma(\triv,\critexp)}, \bigl\langle \psi^{(n)}_{\triv,\critexp} , \psi^{(n)}_{\triv,\critexp} \bigr\rangle_{\LtwoGmodGamma(\triv,\critexp)}  \xrightarrow{n\to\infty} 1,
			\qquad \text{and} \qquad
			\bigl\langle \psi^{(n)}_{\triv,\critexp} , \phi^{(n)}_{\triv,\critexp} \bigr\rangle_{\LtwoGmodGamma(\triv,\critexp)} \xrightarrow{n\to\infty} 0.
		\end{align}
		For sufficiently large $n \in \N$, the set $\{\phi^{(n)}, \psi^{(n)}\}$ is linearly independent and we may define the $2$-dimensional subspace generated by it, $V^{(n)} := \operatorname{span}\set{\phi^{(n)}, \psi^{(n)}} \subset \compactsmooth(\GmodGamma)\cap \LtwoGmodGamma_\t$.
		Further define $Q^{(n)}_{\mrm{rep}}$ and $Q^{(n)}_{\mrm{mix}}$ to be the restrictions of the sesquilinear forms $Q_{\mrm{rep}}$ and $Q_{\mrm{mix}}$ to $V^{(n)}$, respectively.
		It follows from \cref{eqn: Q_rep on critexp spherical complementary series,eq:vectors witnessing multiplicity} and continuity of the determinant that for sufficiently large $n \in \N$, the sesquilinear form $Q^{(n)}_\mrm{rep}$ is non-degenerate and hence has rank $2$.
		
		On the other hand, the sesquilinear form $Q_{\mrm{mix}}$ has rank at most $1$ since it is given by a \emph{scalar} product of two linear forms. Consequently, for sufficiently large $n \in \N$, the same holds for the restriction $Q^{(n)}_{\mrm{mix}}$. This contradicts the fact that $Q^{(n)}_{\mrm{rep}} = Q^{(n)}_{\mrm{mix}}$ for all sufficiently large $n \in \N$ by \cref{eq:Q_rep=Q_mix} and finishes the proof of Part~\cref{itm:simplicity of bottom of spectrum}.

		We now turn to Part~\cref{itm:non-sphericals do not occur at bottom}.
		Recall the orthogonal decomposition $\Lcal_\t =\Lcal_\t^M \oplus \Lcal_\t^0$ into a subspace of $M$-invariant vectors and its orthocomplement in the notation introduced above \cref{eqn: M invariant projection}.
		Since the disintegration of the measures $\BR$ and $\BRstar$  along fibers of the projection $\GmodGamma \r \GmodGamma /M$ is given by our fixed Haar measure on $M$, we have
		\begin{align*}
			Q_{\mrm{mix}}(\phi,\psi) =0, \qquad
			\text{for all } \phi,\psi \in \compactsmooth(\GmodGamma) \cap \LtwoGmodGamma^0_\t.
		\end{align*}
		Now, $\compactsmooth(\GmodGamma) \cap \Lcal^0_\tau $ is a dense subspace of $\Lcal^0_\tau$ by \cref{lem: compact dense in tau type} and continuity of the orthogonal projection $\Lcal_\t \r \Lcal_\t^0$. Moreover, $Q_{\mrm{rep}}$ is continuous on $\LtwoGmodGamma_\tau$ due to the fact that both $\mathsf{T}_\t^\t$ and $\mathsf{C}_+(\critexp)$ are bounded operators since $\critexp > d/2$. Combined with \cref{eq:Q_rep=Q_mix}, this implies
		\begin{align*}
			Q_{\mrm{rep}}(\phi,\psi) = 0 \qquad
			\text{for all $\phi, \psi \in \LtwoGmodGamma^0_\t$}.
		\end{align*}
		Moreover, $ \{\LtwoGmodGamma(\sigma, \critexp)^0_\t\}_{ \sigma \in \MhatSD}$ is a set of mutually orthogonal subspaces. We conclude, in particular, that for all $\sigma \in \MhatSD$ contained in $\tau$, the operator $\mathsf{T}_\t^\t \mathsf{C}_+(\critexp)$ vanishes on $\LtwoGmodGamma(\sigma, \critexp)_\tau^0$.
		Moreover, \cref{lem:T_tau preserves M-0 decomp} also shows that this operator vanishes on $\LtwoGmodGamma(\sigma, \critexp)_\tau^M$ whenever $\sigma$ is non-trivial, and hence on all of $\LtwoGmodGamma(\sigma,\critexp)_\t$ in this case.
		Specializing further, this holds for $\s=\upsilon$, contradicting our choice of the $K$-type $\t$ due to \cref{cor:nonvanishing} and completing the proof.
	\end{proof}

	\subsection{Proof of Theorem~\ref{thm:upgraded decay of correlations for Haar}}
	We now prove \cref{thm:upgraded decay of correlations for Haar} using the established \cref{thm:SSG conjecture} on strong spectral gap. We require \cite[Theorem 4.8]{EdwardsOh} of Edwards--Oh which is derived from another theorem of theirs quoted earlier as \cref{thm:matrix coeff expand}.
	
	\begin{thm}[{\cite[Theorem 4.8]{EdwardsOh}}]
		\label{thm:matrix coeff expand 2}
		There exists $m > d(d + 1)/2$ such that for all $\sigma \in \MhatSD$ and $s \in \Ical_\sigma$, if the quasi-complementary series $\Ucal(\sigma, s)$ contains non-zero $M$-invariant vectors, then for all $u, v \in \Sob_K^m(\GmodGamma)$ and $t \geq 0$, we have
		\begin{align*}
			e^{(d - \critexp)t}\langle \pi_{\sigma, s}(a_t)u, v\rangle_{\Ucal(\sigma, s)}
			=  \sum_{\tau_1, \tau_2 \in \widehat{K}} \langle \mathsf{T}_{\tau_1}^{\tau_2} \mathsf{C}_+(s) \mathsf{P}_{\tau_1} u, \mathsf{P}_{\tau_2} v \rangle_{\Ucal(\sigma, s)}
			+ O_s\bigl(e^{-\eta_s t}\|u\|_{\Sob_K^m(\GmodGamma)} \|v\|_{\Sob_K^m(\GmodGamma)}\bigr).
		\end{align*}
		Here, the sum
		\begin{align*}
			\sum_{\tau_1, \tau_2 \in \widehat{K}} \langle \mathsf{T}_{\tau_1}^{\tau_2} \mathsf{C}_+(s) \mathsf{P}_{\tau_1} u, \mathsf{P}_{\tau_2} v \rangle_{\Ucal(\sigma, s)}
		\end{align*}
		converges absolutely.
	\end{thm}
	
	\begin{remark}
		Though originally not included, \cite[Theorem 4.8]
		{EdwardsOh} of Edwards--Oh also holds for the \emph{ends} of complementary series $\Ucal(\sigma, s)$ for $\sigma \in \MhatSD$ and $s = d - \ell(\sigma) \in \partial\Ical_\sigma$ thanks to the lifting argument in the proof of \cref{cor:matrix coeff expand}. 
	\end{remark}
	
	\begin{proof}[Proof of \cref{thm:upgraded decay of correlations for Haar}]
		Let $\kappa_0:=\min\{\kappa_\Gamma, 1\}>0$ be as in the theorem, where $\kappa_\Gamma$ is the strong spectral gap parameter of $\Gamma$ defined in \cref{def:spec gap parameter}.
		Let
		\begin{align*}
			s_1:=\critexp - \kappa_0.
		\end{align*}
		We use the same decompositions as in the proof of \cref{prop:SSG conjecture part 1} with $s_\star$ replaced by the optimal parameter $s_1$.
		In particular, as in the proof of \cref{prop:SSG conjecture part 1}, we have the following decomposition:
		\begin{align*}
			\LtwoGmodGamma =
			\LtwoGmodGamma_{\temp} \oplus \LtwoGmodGamma(s_1) \oplus \Ucal(\triv, \critexp).
		\end{align*}
		Denote $\LtwoGmodGamma_1 = \LtwoGmodGamma_{\temp} \oplus \LtwoGmodGamma(s_1)$. Let $\phi, \psi \in \Sob_K^m(\GmodGamma)$ and denote their orthogonal decompositions by
		\begin{align*}
			\phi &= \phi_1 + \phi_{\triv, \critexp} \in \LtwoGmodGamma_1 \oplus \Ucal(\triv, \critexp), & \psi &= \psi_1 + \psi_{\triv, \critexp} \in \LtwoGmodGamma_1 \oplus \Ucal(\triv, \critexp).
		\end{align*}
		Exactly as in \cref{eqn: matrix coeff bounds LtwoGmodGammastar}, applying \cref{prop:MatrixDecay} for $\LtwoGmodGamma_1$, we have
		\begin{align*}
			|\langle \phi_1 \circ a_t, \psi_1 \rangle_{\LtwoGmodGamma}| \ll_{\phi_1, \psi_1} (1 + t)e^{-(d - s_1)t} \qquad \text{for all $t > 0$}.
		\end{align*}
		It remains to treat the orthogonal direct summand $\Ucal(\triv, \critexp)$. Clearly, as $\Ucal(\triv, \critexp)$ is spherical, it contains non-zero $M$-invariant vectors. Therefore, we may apply \cref{thm:matrix coeff expand 2} to $\Ucal(\triv, \critexp)$ to obtain for all $t > 0$
		\begin{multline*}
			e^{(d - \critexp)t}\langle \pi_{\triv, \critexp}(a_t)\phi_{\triv, \critexp}, \psi_{\triv, \critexp}\rangle_{\Ucal(\triv, \critexp)} =  \sum_{\tau_1, \tau_2 \in \widehat{K}} \langle \mathsf{T}_{\tau_1}^{\tau_2} \mathsf{C}_+(s) \mathsf{P}_{\tau_1} \phi_{\triv, \critexp}, \mathsf{P}_{\tau_2} \psi_{\triv, \critexp} \rangle_{\Ucal(\triv, \critexp)} \\
			+ O\bigl(e^{-\eta_{\critexp}t}\|\phi_{\triv, \critexp}\|_{\Sob_K^m(\GmodGamma)} \|\psi_{\triv, \critexp}\|_{\Sob_K^m(\GmodGamma)}\bigr),
		\end{multline*}
		where the sum in the first term converges absolutely. As in the proof of \cref{prop:SSG conjecture part 1}, taking $t\r+\infty$, and applying \cref{thm:decay of correlations for Haar}, we see that the sum in the first term is in fact equal to $\int_{\GmodGamma} \phi \, d\BR \int_{\GmodGamma} \overline{\psi} \, d\BRstar$.
		
		Thus, combining these two estimates, we obtain a decay rate of $\min\{ d-s_1,\eta_{\critexp} \} -\e$ for every $\e>0$. Note that we have the trivial inequality $d - s_1 \geq \critexp - s_1 = \kappa_0$. We also have $\eta_{\critexp} = \min\{2\critexp - d, 1\}$ by definition.
		Note further that $s_1\geq d/2$ by definition of $\kappa_\Gamma$; cf.~\cref{def:spec gap parameter}.
		Together with $\critexp > d/2$, we have $2\critexp - d > \critexp - d/2 \geq \critexp - s_1 = \kappa_0$, and hence $\eta_{\critexp} \geq \kappa_0$. This shows that $\min\{ d-s_1,\eta_{\critexp} \} \geq \kappa_0$ and completes the proof.
	\end{proof}

	\bibliographystyle{amsalpha}
	\bibliography{bibliography}
\end{document}